\documentclass{lms}
\usepackage{amssymb}
\usepackage{amsmath}
\usepackage{mathabx}
\usepackage{hyperref}
\usepackage{graphicx}
\usepackage{mathrsfs}
\usepackage[curve]{xypic}
\usepackage{tikz}
\usetikzlibrary{arrows,decorations,shapes,shadows}
\usepackage{verbatim}
\usepackage[normalem]{ulem}
\newcommand{\grassman}{\mathbf G}

\let\cal\mathcal
\renewcommand{\setminus}{\smallsetminus}

\newcommand\R{{\mathbb R}}
\newcommand\C{{\mathbb C}}
\newcommand\Z{{\mathbb Z}}
\newcommand\N{{\mathbb N}}

\newtheorem{theorem}{Theorem}[section]
\newtheorem{proposition}[theorem]{Proposition}

\newtheorem{lemma}[theorem]{Lemma}

\newnumbered{amalgamation}[theorem]{Amalgamation}
\newnumbered{example}[theorem]{Example}
\newnumbered{remark}[theorem]{Remark}
\newnumbered{definition}[theorem]{Definition}

\renewcommand{\int}{\mathop{\rm int}}

\newcommand{\norm}[1]{\lVert #1 \rVert}
  \usepackage{color}
  \usepackage{color}
  \usepackage{color}
  \usepackage{color}
  \usepackage{color}

\title[Characterization of Lipschitz normal embedding]{A characterization of Lipschitz normally embedded surface singularities}
\author{Walter D Neumann, Helge M\o ller Pedersen, Anne Pichon}
\classno{14B05 (primary), 32S25, 32S05, 57M99 (secondary)}
\begin{document}
\maketitle

\begin{abstract} 
  Any germ of a complex analytic space  is  equipped with two natural metrics: the {\it outer metric}  induced by the hermitian metric of the ambient space and the {\it inner metric}, which is the associated riemannian metric on the germ. These two metrics  are in general nonequivalent up to bilipschitz homeomorphism.  We give   a necessary and sufficient condition for a normal  surface singularity to be Lipschitz normally embedded (LNE), i.e., to have bilipschitz equivalent  outer and inner metrics.
  In a partner paper \cite{NPP2} we apply it to prove that rational surface singularities are LNE if and only if they are minimal. 
\end{abstract}

\section{Introduction}

If $(X,0)$ is a germ of a
complex variety, then any embedding $\phi\colon(X,0)\hookrightarrow
(\C^n,0)$ determines two metrics on $(X,0)$: the outer metric
$$d_o({x,y}):=\norm{\phi(x)-\phi(y)}\quad\text{
    (i.e., distance in $\C^n$)}$$ and the inner metric
$$d_i(x,y):=\inf\{\mathop{\rm length}(\phi\circ\gamma):
\gamma\text{ is a rectifyable path in }X\text{ from }x\text{ to
}y\}\, ,$$
using the riemannian metric on $X \setminus \{0\}$ induced by the hermitian metric on $\C^n$. For all $x,y \in X, \ d_o(x,y) \leq d_i(x,y)$, and the outer metric determines the inner metric. Up to
bilipschitz local homeomorphism both these metrics are independent of the
choice of complex embedding.  We speak of the (inner or outer)
\emph{Lipschitz geometry} of $(X,0)$ when considering these metrics up
to bilipschitz equivalence.

\begin{definition} A germ of a complex normal variety $(X,0)$ is \emph{Lipschitz normally embedded} (LNE) if  the identity map of $(X,0)$ is a bilipschitz homeomorphism between inner and outer metrics,  i.e.,  there exists a neighborhood $U$ of $0$ in $X$ and a constant $K\geq 1$ such that for all $x,y \in U$
$$\frac{1}{K} d_i(x,y) \leq d_o(x,y).$$
\end{definition}
 This definition was  first introduced by Birbrair and Mostowski in
  \cite{BM}, where they just call it {\it normally embedded}. We  prefer adding   the word {\it Lipschitz} to distinguish this notion
  from that of projective normal embedding (in algebraic
  geometry) and normality (in local geometry, commutative algebra and
  singularity theory).

 Lipschitz Normal Embedding (LNE) is a very active research area with many
 recent results giving necessary conditions for LNE in the real
 and complex setting, e.g., by Birbrair, Fernandes, Kerner,
 Mendes, Nun\~o-Ballesteros,  Pedersen, Ruas, Sampaio
  (\cite{BMN}, \cite{FS}, \cite{KMR}, \cite{MR})  including a characterization of LNE  for semialgebraic sets (\cite{BM1}). In this paper we focus on complex normal surfaces.

It is a classical fact that the topology of a germ of a complex variety $(X,0)\subset(\C^n,0)$ is locally homeomorphic to the cone over its link $X^{(\epsilon)}=\mathbb S^{2n-1}_{\epsilon} \cap X$, where $\mathbb S^{2n-1}_{\epsilon}$ denotes the sphere with small radius $\epsilon$ centered at the origin in $\C^n$. 
If $(X,0)$ is a curve germ  then  it  is in fact bilipschitz equivalent to the metric cone over its link with respect to the inner metric, while the data of  its Lipschitz outer geometry is equivalent to that of the embedded topology of a generic plane projection (see \cite{PT},  \cite{NP1}).  Therefore,  an irreducible complex curve is LNE if and only if it is smooth. This is not true in higher dimension.

The main result of this paper, Theorem \ref{cor:complex characterization of normal embedding}, is a characterization of LNE for normal surface germs based on what we call the {\it nodal test curve criterion}.  
        The first ingredient of the proof  is the {\it arc criterion} for LNE  (Theorem \ref{thm:arc criterion})  of Birbrair and Mendes \cite{BM1} which says that one 
  can check if a semialgebraic germ $(X,0)$ is LNE by testing LNE on each
  pair of real analytic arc $\delta_1,\delta_2\in (X,0)$. This criterion is difficult to use effectively since the amount of pairs of curves is incommensurable.  Our Theorem 
  \ref{cor:complex characterization of normal embedding} uses  \emph{generic
    projections} $\ell \colon (X,0) \to (\C^2,0)$ which enable one to  reduce drastically  the amount of types of pairs of 
  arcs to be tested, namely just certain pairs of arcs in $\ell^{-1}(\delta)$ for
  certain arcs $\delta \subset (\C^2,0)$ called {\it test arcs}. This makes
the criterion more efficient to prove LNE.  For example, we use it in \cite{NPP2} to prove that rational surface
  singularities are LNE if and only if they are minimal.
  
  The second ingredient of the proof is the geometric decomposition of a normal surface germ which was introduced in \cite{BNP} and which is presented in Section  \ref{sec:geometric decomposition}.
  
  The proof  of Theorem \ref{cor:complex characterization of normal embedding} has two keystones:   Proposition \ref{prop:characterization of normal embedding2} and  its enhancement Proposition \ref{cor:characterization of normal embedding}. They  consist of two successive reductions of the amount of test arcs to be tested. They are stated in terms of real test arcs which are real slices of complex curves on $(X,0)$.  Then the final part of the proof (Section \ref{sec:proof main}) consists of a reinterpretation of Proposition \ref{cor:characterization of normal embedding} in terms of complex test curves. 

\vskip.1cm\noindent{\bf Acknowledgments.}  
Neumann was supported by NSF grant DMS-1608600.  Pedersen was supported
by FAPESP grant 2015/08026-4. Pichon was supported by the ANR project
LISA  17-CE40--0023-01 and by USP-Cofecub UcMa163-17. We are  very grateful for the hospitality and
support of the following institutions: Columbia University, Institut
de Math\'ematiques de Marseille, FRUMAM Marseille, Aix Marseille
Universit\'e, ICMC-USP and IAS Princeton.

\section{Generic projections and Nash modification}

In order to state the main Theorem  \ref{cor:complex characterization of normal embedding}  in Section \ref{statement}, we need to introduce the notions of generic projections of a curve and of a surface, and of Nash modification. 
  
Let $\cal D$ be a $(n-2)$-plane in $\C^n$ and let $\ell_{\cal D}
\colon \C^n \to \C^2$ be the linear projection with
kernel $\cal D$. Suppose $(C,0)\subset (\C^n,0)$ is a complex curve germ.  There exists an open
dense subset $\Omega_C$ of the Grassmanian $\grassman(n-2,\C^n)$  such that for $\cal D \in \Omega_C$, $\cal D$ contains no limit of secant lines to the curve $C$ (\cite{teissier}). 
\begin{definition} \label{def:generic projection curve}  The
  projection $\ell_{\cal D}$ is  said  to be \emph{generic for $C$} if  $\cal D \in \Omega_C$. 
\end{definition}
 In the sequel, we will use extensively the following result 
\begin{theorem}[{\cite[pp. 352-354]{teissier}}]  \label{generic projection bilipschitz}If $\ell_{\cal D}$ is a generic projection for $C$, then the restriction $\ell_{\cal D}|_{C} \colon C \to \ell_{\cal D}(C)$ is a bilipschitz homeomorphism for the outer metric.
\end{theorem} 

 Let $(X,0)\subset (\C^n,0)$ be a normal surface singularity. We restrict ourselves to those $\cal D$ in $\grassman(n-2,\C^n)$ such that the restriction
$\ell_{\cal D}{\mid_{(X,0)}}\colon(X,0)\to(\C^2,0)$ is finite.
The \emph{polar curve}
$\Pi_{\cal D}$ of $(X,0)$ for the direction $\cal D$ is the closure in
$(X,0)$ of the singular locus of the restriction of $\ell_{\cal D} $
to $X \setminus \{0\}$. The \emph{discriminant curve} $\Delta_{\cal D}
\subset (\C^2,0)$ is the image $\ell_{\cal D}(\Pi_{\cal D})$ of the polar
curve $\Pi_{\cal D}$.

\begin{proposition}[{\cite[Lemme-cl\'e V 1.2.2]{teissier}}]\label{prop:generic} An open
dense subset $\Omega \subset \grassman(n-2,\C^n)$ exists 
such that: 
\begin{enumerate}
\item \label{cond:generic1} the family of  curve germs  $(\Pi_{\cal D})_{\cal D \in \Omega}$
is equisingular in terms of strong simultaneous resolution;
\item \label{cond:generic2} the discriminant curves  $ \Delta_{\cal D}=\ell_{\cal D}(\Pi_{\cal D})$, ${\cal D \in \Omega}$, form an equisingular family of reduced plane curves;
\item \label{cond:generic3}  for each $\cal D$, the projection $\ell_{\cal D}$ is generic for its polar curve $\Pi_{\cal D}$. 
\end{enumerate}
  \end{proposition}

\begin{definition}  \label{def:generic linear projection} 
The projection $\ell_{\cal D} \colon \C^n \to \C^2$
  is \emph{generic} for $(X,0)$ if $\cal D \in \Omega$.
\end{definition}

\begin{remark} \label{rk:generic}
  \begin{enumerate}
\item[(1)]  Conditions \ref{cond:generic1} and \ref{cond:generic3} are  explicitly stated in \cite[Lemme-cl\'e V 1.2.2]{teissier}. Condition \ref{cond:generic2}, whose openness is proved in \cite[Chap. I, 6.4.2]{teissier},  appears implicitly in \cite[Lemme-cl\'e V 1.2.2]{teissier}  since it is used in its proof  as a pre-condition reducing $\Omega$ to obtain Condition  \ref{cond:generic3} (see \cite{teissier} bottom of page 463).
\item[(2)] Conditions \ref{cond:generic2} and \ref{cond:generic3} imply that  the family of plane curves
    $\ell_{\cal D}(\Pi_{\cal D'} )$ parametrized by
    $(\cal D,\cal D') \in \Omega \times \Omega$ is equisingular on a
    Zariski open neighborhood of the diagonal
    $\{(\cal D, \cal D) :\cal D \in \Omega \}$.
\end{enumerate}
\end{remark}

\begin{definition}[(Nash modification)]\label{def:Nash modification}
Let $\lambda \colon X\setminus\{0\} \to \grassman(2,\C^n)$ be the
  map which sends $x \in X\setminus\{0\}$ to the tangent plane
  $T_xX$. The closure $N X$ of the graph of $\lambda$ in $X
  \times \grassman(2,\C^n)$ is a reduced analytic surface. By
  definition, the \emph{Nash modification} of $(X,0)$ is the induced
  morphism $\mathscr N  \colon NX \to X$. 
\end{definition}

\begin{lemma}[{\cite[Part III,  Theorem
    1.2]{S}}, {\cite[Section 2]{GS}}]\label{le:nash}
A morphism $f \colon Y \to X$ factors through Nash modification if and
  only if it has no base points for the family of polar curves.
  \end{lemma}

\section{Statement of the  theorem}\label{statement}
  
  Let  $\ell \colon (X,0) \to (\C^2,0)$ be a generic projection, let $\Pi $ be its polar curve and let $\Delta = \ell(\Pi)$ be its discriminant curve. Denote by $\rho'_{\ell} \colon Y_{\ell} \to \C^2$ the minimal composition of blow-ups of points starting with the blow-up of the origin which resolves the base points of the family of projections of generic polar curves  $(\ell(\Pi_{\cal D}))_{\cal D \in \Omega}$.

\begin{definition} \label{def:Delta-curve}
 We say \emph{$\Delta$-curve} for an exceptional curve in
    {$(\rho'_{\ell})^{-1}(0)$} intersecting the strict transform of $\Delta$.  
    \end{definition}
  
   Let us blow up all the intersection points between two
   $\Delta$-curves. We call $\sigma \colon Z_{\ell} \to Y_{\ell}$ and
   $\rho_{\ell}=\rho'_\ell\circ \sigma\colon Z_\ell\to\C^2$ the
   resulting morphisms (if no $\Delta$-curves intersect,
   $\rho_ {\ell} =  {\rho'_{\ell}}$).  
  
 \begin{remark}\label{rem:notations}  By (2) of remark \ref{rk:generic}, the resolution graph of $\rho_{\ell}$ does not depend on $\ell$. We denote it by  $T$ for the rest of the paper.  
 \end{remark}

 \begin{definition} \label{def:node T} A  {\it $\Delta$-node} of $T$ is a vertex $(j)$ of $T$ which represents a $\Delta$-curve. If two $\Delta$-nodes are joined by a string of valency-two vertices which contains neither a $\Delta$-node nor the root vertex, we choose a vertex on that string, and we call it a {\it separation-node} (in particular, a vertex $(j)$ associated with an exceptional curve $C_j$ resulting from the blow-ups $\sigma\colon Z_\ell\to Y_\ell$ is a separation node). 
  
  A {\it node} of $T$ is  a vertex $(j)$ of $T$  which is either  the root-vertex or a  $\Delta$-node or a separation-node or a vertex with at least three incident edges.
  \end{definition}
 
Let $E \subset Y$ be a complex curve in a complex surface $Y$ and let $E_1,\ldots,E_n$ be the irreducible components of $E$. We say {\it curvette} of $E_i$ for any smooth curve germ $(\beta,p)$ in $Y$, where $p$ is a point of $E_i$ which is a smooth point of $Y$ and $E$ and such that  $\beta$ and $E_i$ intersect transversely. 

If $G$ is a graph, we will denote by $V(G)$ its set of vertices and by $E(G)$ its set of edges.

\begin{definition}\label{def:test curve}  Let  $C_i$  be the irreducible component of $\rho_{\ell}^{-1}(0) $ represented by $(i) \in V({T})$, so we have $\rho_\ell^{-1}(0) = \bigcup_{(i) \in V({T})} C_i$. For $(i) \in V({T})$ we call \emph{test curve at $(i)$} (of $\ell$)  any complex curve germ $(\gamma,0) \subset (\C^2,0)$ such that 
\begin{enumerate}
\item the strict transform $\gamma^*$   by  $\rho_\ell$ is a curvette of a $C_i$;
\item $\gamma^* \cap \Delta^* = \emptyset$. 
\end{enumerate}
 A test-curve $\gamma \subset (\C^2,0)$ at $(i)$ is called a {\it nodal test curve} if $(i)$ is a node of $T$. 
\end{definition}

\begin{definition}\label{def:node G resolution}  Let $\pi_0 \colon X_0 \to X$ be the minimal good  resolution of $X$ which factors through both the Nash modification and the blow-up of the maximal ideal and let $G_0$ be its resolution graph. For each vertex $(v)$ of $G_0$  we denote by $E_v$ the corresponding  irreducible component of $\pi_0^{-1}(0)$. A vertex $(v)$ of $G_0$ such that $E_v$ is an irreducible component of the blow-up of the maximal ideal (resp.\ an exceptional curve of the Nash transform)  is called an {\it $\cal L$-node} (resp.\ a {\it $\cal P$-node}) of $G_0$. 
\end{definition}

\begin{definition}\label{def:G'} Consider the graph $G'_0$ of $G_0$ defined as the union of all simple paths in  $G_0$ connecting  pairs of  vertices among $\cal L$- and $\cal P$-nodes. Let $\ell \colon (X,0) \to (\C^2,0)$ be a generic projection.
  Let  $\gamma$ be a test curve for $\ell$. A component $\widehat{\gamma}$ of $\ell^{-1}(\gamma)$ is called {\it principal} if its  strict transform by $\pi_0$ is either a curvette of a component $E_v$ with $v \in V(G'_0)$ or intersects $\pi_0^{-1}(0)$ at an intersection between two exceptional curves  $E_v$ and $E_{v'}$ such that both  $(v)$ and $(v')$ are  in $V(G'_0)$.
\end{definition}

We now define the outer and inner contacts between two complex curves on a complex surface germ.

 Throughout the paper, we use the ``big-Theta" asymptotic notation of Bachman-Landau:  given two function germs $f,g\colon ([0,\infty),0)\to ([0,\infty),0)$ we say $f$ is \emph{big-Theta} of $g$ and we write   $f(t) = \Theta (g(t))$ if there exist real numbers $\eta>0$ and $K \ge 1$ such that for all $t$ with $f(t)\le\eta$: $$\frac{1}{K }g(t) \leq f(t) \leq K g(t).$$

  Let $\mathbb S^{2n-1}_{\epsilon} = \{ x \in \C^n \colon \norm x_{\C^n} = \epsilon\}$.  Let $(\gamma_1,0)$ and $(\gamma_2,0)$ be two germs of complex curves inside $(X,0)$. Let $q_{out}=q _{out}(\gamma_1, \gamma_2)$ and $q_{inn}=q_{inn}(\gamma_1, \gamma_2)$ be the two rational numbers $\geq 1$ defined by 
$$ d_o(\gamma_1 \cap \mathbb S^{2n-1}_{\epsilon}, \gamma_2 \cap \mathbb S^{2n-1}_{\epsilon}) =  \Theta(\epsilon^{q_{out}}),$$
$$ d_i(\gamma_1 \cap \mathbb S^{2n-1}_{\epsilon}, \gamma_2 \cap \mathbb S^{2n-1}_{\epsilon}) =  \Theta(\epsilon^{q_{inn}}),$$
where $d_{i}$ means inner distance in $(X,0)$ as before. (The existence and  rationality of $q_{inn}$ will be a consequence of Proposition \ref{lem:complex real rates}). 

\begin{definition} We call $q _{out}(\gamma_1, \gamma_2)$ (resp.\ $q _{inn}(\gamma_1, \gamma_2)$) the {\it outer  contact exponent} (resp.\ {\it  the inner  contact exponent}) between $\gamma_1$ and $\gamma_2$.
\end{definition}


We now state the main  result of the paper.

\begin{theorem}[{(Test curve criterion for LNE of a complex surface)}] \label{cor:complex characterization of normal embedding} A normal surface germ  $(X,0)$  is LNE if and only if the following conditions are satisfied for all generic projections $\ell \colon (X,0) \to (\C^2,0)$ and nodal test curves $(\gamma,0) \subset (\C^2,0)$: 
\begin{itemize}
\item[($1^*$)] \label{iso} for all  principal components $\widehat{\gamma}$ of $\ell^{-1}(\gamma)$, $mult(\widehat{\gamma})=mult({\gamma})$ where $mult$ means multiplicity at $0$; 
\item[($2^*$)] \label{vertical} for all pairs $(\gamma_1,\gamma_2)$ of distinct
principal components of $\ell^{-1}(\gamma)$,
  $q_{inn}(\gamma_1,\gamma_2) = q_{out}(\gamma_1,\gamma_2)$.
\end{itemize}
\end{theorem}

\begin{remark} The definition of separation node (Definition
  \ref{def:node T}) depends on a choice of a vertex along a string
  joining two $\Delta$-nodes. However, the validity of Theorem
  \ref{cor:complex characterization of normal embedding} does not
  depend on this choice. In fact it follows from the proof that if all
  the principal components over one test curve at a vertex on such a
  string satisfy ($1^*$) and ($2^*$) then all principal components
  over all test curves at vertices on the string satisfy ($1^*$) and
  ($2^*$).
\end{remark}

\section{The real arc criterion for LNE of a semialgebraic germ}

\begin{definition}[(Real arcs)]
     Let $(X,0) \subset (\R^n,0)$ be a semialgebraic germ.  A \emph{real arc} on $(X,0)$ will mean the germ of a semialgebraic map $\delta \colon [0,\eta) \to X$ for some $\eta \in \R^+$, such that $\delta(0)=0$ and $\norm{\delta(t)}=t$ (see also Remark \ref{rk:parametrization}).

     When no confusion is possible, we will use the same notation for the arc $\delta$  and the germ of its parametrized image $\delta([0,\eta))$. 
 \end{definition}

\begin{definition} Let $(X,0) \subset (\R^n,0)$ be a semialgebraic germ and let $\delta_1 \colon [0,\eta)  \to X$ and $\delta_2  \colon [0,\eta)  \to X$  be two real arcs on $X$. The {\it outer contact} of $\delta_1$ and $\delta_2$  is  $\infty$ if $\delta_1=\delta_2$ and is otherwise the rational number $q_{o}=q_{o}(\delta_1,\delta_2)$ defined by:
$$\norm{\delta_1(t)-\delta_2(t)} = \Theta(t^{q_{o}}).$$

The {\it inner contact} of $\delta_1$ and $\delta_2$  is the rational number $q_{i}=q_{i}(\delta_1,\delta_2)$ defined by
$$d_{i}(\delta_1(t),\delta_2(t)) = \Theta(t^{q_i}).$$
\end{definition}

\begin{remark}\label{rk:parametrization} 1)~~The existence and rationality of $q_i$ comes from the fact that there exists a semialgebraic metric $d_P \colon X \times X \to \R$ (the so-called pancake metric)  such that $d_i$ and $d_P$ are bilipschitz equivalent (\cite{KO}).  
 
2) The inner and outer contacts $q_{i}(\delta_1,\delta_2)$ and $q_{o}(\delta_1,\delta_2)$ can also be defined taking reparametrizations by real slices of $\delta_1$ and $\delta_2$ as follows. First note that if $\delta_1$ and $\delta_2$ have different tangent directions then $q_i(\delta_1,\delta_2)=q_o(\delta_1,\delta_2)=1$, so we may assume they have the same tangent direction. We can choose coordinates $(x_1,\ldots,x_n)$ such that the tangent semi-line of $\delta_1$ and $\delta_2$ has $x_1>0$ except at $0$. For $j=1,2$, consider the  reparametrization $\widetilde{\delta}_j \colon [0,\eta) \to \R^n$ defined by   $\widetilde{\delta}_j(t) = \delta_j \cap \{x_1=t\}$. Then we have $\norm{\widetilde{\delta}_1(t)-\widetilde{\delta}_2(t)} = \Theta(t^{{q_o}}) $ and $d_{i}(\widetilde{\delta}_1(t),\widetilde{\delta}_2(t)) = \Theta(t^{{q_i}})$.

Indeed, this is an easy consequence of the following standard lemma:
\begin{lemma}Let $B\subset \C^n$ be any closed compact convex
    neighborhood of $0$ in $\C^n$. Let $\phi\colon B\to B_1$, where
    $B_1$ is the unit ball, be the homeomorphism which maps each ray
    from $0$ to $\partial B$ linearly to the ray with the same
    tangent, but of length $1$. Then the map $\phi\colon B\to B_1$ is
    a bilipschitz homeomorphism.
\end{lemma}
\end{remark}
 
The following  result of Birbrair and Mendes is a characterization of closed semialgebraic germs which are LNE. 

\begin{theorem}[(The arc criterion for LNE, \cite{BM1})] \label{thm:arc criterion}
Let $(X,0)\subset (\R^m,0)$ be a closed semialgebraic germ. It is LNE if and only if for all pairs of real arcs $\delta_1$ and $\delta_2$ in $(X,0)$, 
$q_i(\delta_1,\delta_2) = q_o(\delta_1,\delta_2)$.
\end{theorem}


\section{ The real test arc criterion for LNE of a normal complex surface germ} \label{sec:test curve criterion}

 \begin{definition} Let $(\gamma,0) \subset (\C^n,0)$ be a complex curve germ. Denote by $z_1,\ldots, z_n$ the coordinates of $\C^n$.  We assume that no tangent line to $\gamma$ is contained in the hyperplane $\{z_1=0\}$.  Fix  $e^{i \alpha} \in \Bbb S^1$.  We call the intersection $\gamma_{\alpha} = \gamma \cap \{z_1=e^{i \alpha} t, t \in \R^+ \}$ a {\it real slice} of $\gamma$. 
\end{definition} 

 \begin{definition}  \label{def:test arc}  
We call \emph{test  arc }   a  component of a real slice of a test curve $(\gamma,0) \subset (\C^2,0)$ (Definition \ref{def:test curve}). 
\end{definition}

  \begin{proposition} \label{prop:characterization of normal embedding2} A normal surface  $(X,0)$  is LNE if and only if  for all  generic projections $\ell \colon (X,0) \to (\C^2,0)$ and for all test arcs  $\delta$ of $\ell$,   any pair of components $\delta_1,\delta_2$ of $\ell^{-1}(\delta)$ satisfies  $q_i(\delta_1,\delta_2) = q_o(\delta_1,\delta_2)$.  
\end{proposition}

Proposition \ref{prop:characterization of normal embedding2} will be proved in Section \ref{sec:proof}. Before this, we introduce the necessary material for the proof in Section \ref{sec:geometric decomposition} and we prove preliminary lemmas in Sections \ref{sec:vertically aligned} to  \ref{sec:strings}. 

  We will also prove later an enhanced version of this result  (Proposition \ref{cor:characterization of normal embedding})  which reduces again drastically  the amount of  pairs of test arcs and which is a version  in terms of real  test arcs  of the main Theorem \ref{cor:complex characterization of normal embedding}).

\section{The geometric decomposition of a normal surface germ} \label{sec:geometric decomposition}

\subsection{Pieces}Birbrair, Neumann and Pichon defined in \cite{BNP} some semialgebraic
metric spaces (with inner metric) called $A$-, $B$- and $D$-pieces,
which we will need in this paper. We give here the
definition and basic facts about these pieces (see \cite[Sections 11
and 13]{BNP} for more details).

   The pieces are topologically
  conical, but usually with metrics that make them shrink non-linearly
  towards the cone point.  We will consider these pieces as germs at
  their cone-points, but for the moment, to simplify notation, we
  suppress this.
  
  By $D^2$ we mean the standard unit disc in $\C$ with $\Bbb S^1$ as its boundary, and $I$ denotes the interval $[0,1]$. 
  
\begin{definition}[\bf$\boldsymbol{A(q,q')}$-pieces]\label{def:Aqq'}
  Let $q,q'$ be rational numbers such that $1\le q  < q'$. Let $A$ be
  the euclidean annulus $\{(\rho,\psi):1\le \rho\le 2,\, 0\le \psi\le
  2\pi\}$ in polar coordinates and for $0<r\le 1$ let $g^{(r)}_{q,q'}$
  be the metric on $A$:
$$g^{(r)}_{q,q'}:=(r^q-r^{q'})^2d\rho^2+((\rho-1)r^q+(2-\rho)r^{q'})^2d\psi^2\,.
$$ 
So $A$ with this metric is isometric to the euclidean annulus with
inner and outer radii $r^{q'}$ and $r^q$. The metric completion of
$(0,1]\times \Bbb S^1\times A$ with the metric
$$dr^2+r^2d\theta^2+g^{(r)}_{q,q'}$$ compactifies it by adding a single point at
$r=0$.  We call a metric space which is bilipschitz homeomorphic to
this completion an \emph{$A(q,q')$-piece} or simply  {an} \emph{$A$-piece}.
 \end{definition}
 
 \begin{definition}[\bf$\boldsymbol{B(q)}$-pieces]\label{def:Bq}  
   Let $F$ be a compact oriented $2$-manifold, $\phi\colon F\to F$ an
   orientation preserving diffeomorphism, and $M_\phi$ the mapping
   torus of $\phi$, defined as:
$$M_\phi:=([0,2\pi]\times F)/((2\pi,x)\sim(0,\phi(x)))\,.$$
Given a rational number   $q > 1$ , we will define a metric space
$B(F,\phi,q)$ which is topologically the cone on the mapping torus
$M_\phi$.
 
For each $0\le \theta\le 2\pi$ choose a Riemannian metric $g_\theta$
on $F$, varying smoothly with $\theta$, such that for some small
$\delta>0$:
$$
g_\theta=
\begin{cases}
g_0&\text{ for } \theta\in[0,\delta]\,,\\
\phi^*g_{0}&\text{ for }\theta\in[2\pi-\delta,2\pi]\,.
\end{cases}
$$
Then for any $r\in(0,1]$ the metric $r^2d\theta^2+r^{2q}g_\theta$ on
$[0,2\pi]\times F$ induces a smooth metric on $M_\phi$. Thus
$$dr^2+r^2d\theta^2+r^{2q}g_\theta$$ defines a smooth metric on
$(0,1]\times M_\phi$. The metric completion of $(0,1]\times M_\phi$
adds a single point at $r=0$.  Denote this completion by $B(F,\phi,q)$. We call a metric space which is bilipschitz homeomorphic to $B(F,\phi,q)$ a  
     \emph{$B(q)$-piece} or simply {a} \emph{$B$-piece}. 

A $B(q)$-piece such that $F$ is a disc   is called a \emph{$D(q)$-piece} or simply
 {a}   \emph{$D$-piece}.
\end{definition}

\begin{definition}[\bf Conical pieces]\label{def:p3}
 Given a compact smooth $3$-manifold $M$,
  choose a Riemannian metric $g$ on $M$ and consider the metric
  $dr^2+r^2g$ on $(0,1]\times M$. The completion of this adds a point
  at $r=0$, giving a \emph{metric cone on $M$}.  We call a metric space which is bilipschitz homeomorphic to a metric cone a \emph{conical piece}.   We will call  any conical piece a \emph{$B(1)$-piece}  (they were called $CM$-pieces in \cite{BNP}).
\end{definition}

The diameter of the image in $B(F,\phi,q)$ of a fiber  $ \{r\}\times \{\theta\}\times F$ is $\Theta(r^q)$. Therefore $q$ describes
a rate of shrink of the surfaces    $\{r\}\times \{\theta\}\times F$ in $B(F,\phi,q)$ with
respect to the distance $r$ to the point at $r=0$. Similarly, the
inner and outer boundary components of any $\{r\} \times \{ t \} \times A$ in an $A(q,q')$ have rate of shrink
respectively $q'$ and $q$ with respect to $r$.

\begin{definition}[\bf Rate]
  The rational number $q$ is called the \emph{rate} of $B(q)$ or  $D(q)$. The rational numbers $q$ and $q'$ are the two \emph{rates} of $A(q,q')$.
\end{definition}

\subsection{Classical plane curve theory} \label{subsec:contact}
  Let $(\gamma,0) \subset (\C^2,0)$ and $(\gamma',0)\subset (\C^2,0)$
be two distinct germs of irreducible complex plane curves. Let
$\mathbb S^{3}_{\epsilon} = \{ x \in \C^2 \colon \norm x =
\epsilon\}$.  Recall that the contact $q_{out}=q _{out}(\gamma, \gamma')$ of
$\gamma$ and $\gamma'$ is the rational number defined by:
$$ d_o(\gamma_1 \cap \mathbb S^{3}_{\epsilon},
\gamma_2 \cap \mathbb S^{3}_{\epsilon}) =  \Theta(\epsilon^{q_{out}}).$$
Equivalently, $q _{out}(\gamma, \gamma')$ is the
 largest $q$ for which   there exist  Puiseux
 expansions $y=f(x)$ and   $y=g(x)$ of $\gamma$ and $\gamma'$  {which} coincide for exponents $<q$ and {which} have distinct  coefficients for $x^q$.

Let $\rho \colon Y \to \C^2$ be a sequence of blow-ups of points starting with the blow-up of the origin of $\C^2$.    Let  $C_1,\dots, C_k$ be the irreducible components in $Y$ of the exceptional divisor of $\rho$ with $C_1$ being the exceptional curve of the first
blow-up. 

Let $(\gamma,0)$ and $(\gamma',0)$ be two irreducible curve germs whose strict transforms by $\rho$ meet a $C_i$ at two distinct smooth points of $\rho^{-1}(0)$. Then the contact $q_{out}(\gamma, \gamma') $ of $\gamma$ and $\gamma'$ in $\C^2$ does not depend on the choice of $\gamma$ and
$\gamma'$.

\begin{definition} \label{def:inner rate} We call
  $ q_{out}(\gamma, \gamma')$ the \emph{inner rate} of $C_i$ and we
  denote it by $q_{C_i}$, or simply $q_i$ when no confusion is possible. 
 \end{definition}
 
Let $T_0$ be the dual tree of $\rho$,  {i.e., the vertices  $(1), \ldots, (k)$ of $T_0$ are in bijection with the exceptional curves $C_1,\ldots,C_k$ and there is an edge between $(i)$ and $(j)$ if and only if  $C_i \cap C_j \neq \emptyset$. The vertex $(1)$, corresponding to the curve $C_1$,  is the root  of $T_0$. We weight  each vertex $(i)$ by the corresponding inner rate $q_i=q_{C_i}$.} 
 
A non-root vertex with one incident edge is called a \emph{leaf}. By classical theory on plane curves, we have:

\begin{proposition} \label{lem:increasing rates}   $q_{1} = 1$ and the inner rates along a
   path from the root to a leaf are strictly increasing.
 \end{proposition} 

\begin{example}\label{example:graphs} Consider the curve germ
 $(\gamma,0)$ with Puiseux expansion $y= f(x)=x^{3/2}+ x^{7/4}$.  Its resolution graph $T_0$, with exceptional curves labeled in order of blow-up, is pictured on the left. Its arrow represents the strict transform $\gamma^*$, which meets $C_5$ at a smooth point.

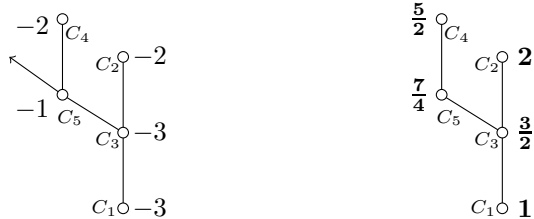
\begin{center}
 \begin{tikzpicture}
   \draw[thin](0,0)--(0,2);
   \draw[thin](0, 1)--(-.8,1.5);
   \draw[thin](-.8,1.5)--(-.8,2.5);
   \draw[thin,>-stealth,->](-.8,1.5)--+(-.7,.5);
   \draw[fill=white](0,0)circle(2pt);\node(a)at(0.35,0){$-3$};\node(a)at(-0.25,0){\scriptsize{$C_1$}};
   \draw[fill=white](0,1)circle(2pt);\node(a)at(0.35,1){$-3$};\node(a)at(-0.2,.9){\scriptsize{$C_3$}};
   \draw[fill=white](0,2)circle(2pt);\node(a)at(0.35,2){$-2$};\node(a)at(-0.2,1.9){\scriptsize{$C_2$}};
   \draw[fill=white](-.8,1.5)circle(2pt);\node(a)at(-1.2,1.3){$-1$};\node(a)at(-.7, 1.2){\scriptsize{$C_5$}};
   \draw[fill=white](-.8,2.5)circle(2pt);\node(a)at(-1.2,2.4){$-2$};\node(a)at(-.6, 2.3){\scriptsize{$C_4$}};

\draw[thin](5,0)--(5,2);
\draw[thin](5,1)--(4.2,1.5);
\draw[thin](4.2,1.5)--(4.2,2.5);
\draw[fill=white](5,0)circle(2pt); 
\node(a)at(5.3,0){{\bf 1}};\node(a)at(4.75,0){\scriptsize{$C_1$}};
\draw[fill=white](5,1)circle(2pt); 
\node(a)at(5.3,1){{$\bf\frac32$}};\node(a)at(4.8,.9){\scriptsize{$C_3$}};
\draw[fill=white](5,2)circle(2pt);
\node(a)at(5.3,2){{\bf 2}};\node(a)at(4.8,1.9){\scriptsize{$C_2$}};
\draw[fill=white](4.2,1.5)circle(2pt); 
\node(a)at(3.9,1.5){{$\bf\frac74$}};\node(a)at(4.3, 1.2){\scriptsize{$C_5$}};
\draw[fill=white](4.2,2.5)circle(2pt);
\node(a)at(3.9,2.5){{$\bf\frac52$}};\node(a)at(4.4,2.3){\scriptsize{$C_4$}};
\end{tikzpicture}
\end{center}
By changing the coefficient of $x^{7/4}$ in $f(x)$ one obtains a
curve germ $\gamma'$ whose strict transform ${\gamma'}^*$ meets at a different
point of $C_5$. The contact exponent of $\gamma$ and $\gamma'$, and hence the
inner rate    {$q_5$} of $C_5$, is $7/4$. Similarly, replacing the coefficient of
$x^{3/2}$ in $f(x)$ by nonzero coefficients other than $1$ gives curves
  whose strict transforms meet $C_3$ at distinct smooth points and
  which have  {pairwise} contact $3/2$.  The inner rates    {$q_i$} for all $C_i$'s are shown in
  the picture on the right.
\end{example}

\subsection{Geometric decomposition and inner contacts in $\C^2$} \label{subsec: geometric decomposition of C^2}
For each $i=1,\dots,k$ let $N(C_i)\subset Y$ be a disk-bundle
neighborhood of $C_i$ and set
  $$\cal N(C_i):=\overline{N(C_i)\setminus \bigcup_{j\ne i} N(C_j)}.$$ 
Let $h\colon \C^2\to \C$ be a linear projection such that the point
$h^*\cap \rho^{-1}(0)\in C_1$ is a smooth point of $C_1$,   {where $^*$ means strict transform by $\rho$}. In the
neigbourhood $U$ of $\rho^{-1}(0)$ given by $|h|<\eta$ with $\eta$
sufficiently small, we have a decomposition of $Y$ into pieces as
follows:\\
--~for each $i=1\dots,k$, $\rho(\cal N(C_i))$ is a $B(q_{i})$-piece $B_i$;\\
--~a $B$-piece corresponding to a leaf of $T_0$ is   a $D$-piece;\\  
--~for each $q_{i}<q_{j}$ with $C_i\cap C_j\ne \emptyset$, $A_{ij}:=\rho(N(C_i)\cap N(C_j))$ is an   $A(q_{i},q_{j})$-piece.

We may assume our generic linear form $h \colon \C^2 \to \C$ has coordinates chosen so the   {$y$-axis} is the kernel of $h$. For $t \in \C$, we set $F_t := \{(x,y):x=t\}$.

 When $q_i \neq 1$, the fibers of the $B({q_i})$-piece  $B_i$ are the intersections   $B_i \cap F_t$ and by Proposition \ref{lem:increasing rates}, each  {connected} component of the fiber   $B_i \cap F_t$  is a  disc with discs inside removed, its  diameter is $\Theta(|t|^{q_i})$, and it shrinks uniformly as $t$ tends to $0$.   An $A(q_i,q_j)$-piece has annular fibers by intersecting with $F_t, t \in \C$.

This induces a decomposition of $F_t  $ as the union  $$F_t =   \bigcup_{i} (B_i \cap F_t) \cup \bigcup (A_{ij} \cap F_t).$$
 Consider the graph $\cal F$ defined as follows: the vertices are in
 bijection with the connected components of  $\bigcup_{i} (B_i \cap
 F_t) $ and  the edges between two vertices are in bijection with the
 annuli of $\bigcup (A_{i,j} \cap F_t)$ between the two corresponding
 components  of $\bigcup_{i} (B_i \cap F_t) $. We weight each vertex
 by the corresponding inner rate  and we  then have a natural
 surjective graph-map $$ \cal E \colon  \cal F \to T_0,$$ i.e., $\cal
   E(V(\cal F)) =V(T_0)$ and the image by $\cal E$ of an edge $(v,v')$
   of $\cal F$ is the edge $(\cal E(v),\cal E(v'))$ of
   $T_0$. Moreover, $\cal E$ preserves the inner rates, i.e., for each
   vertex $(v)$ of $\cal F$ the corresponding inner rate $q_v$
   satisfies: $q_v= q_{\cal E(v)}$. If   $(\nu)$ is a vertex of $\cal F$, we denote by $F_{t,\nu}$ the corresponding component of the decomposition of  $F_t$   {and we denote its inner rate by $q_{\nu}$}.  

\begin{definition} \label{def:geometric decomposition1}
We call the decomposition of  the germ $(\C^2,0)$ as the union of germs: $$ (\C^2,0) = \bigcup_{i \in V(T_0)} B_i \cup  \bigcup_{ {(}i {,j)} \in E(T_0)} A_{i,j}$$
\emph{the geometric decomposition of $(\C^2,0)$ associated to $\rho$}.

We call the decomposition of the real $2$-plane $F_t$ as the union of the   {$ F_{t,\nu}$ for $(\nu) \in V(\cal F)$} and intermediate annuli $F_t \cap A_{i,j}$
\emph{the geometric decomposition of $F_t$ associated to $\rho$} and we call the graph $\cal F$ weighted by inner rates its graph. 
\end{definition}

\begin{example}
 Using example \ref{example:graphs} we show below a schematic picture of $F_t$ next to the graph $\mathcal F$. The gray and white pieces are respectively component slices of $A$- and $B$-pieces.
\\
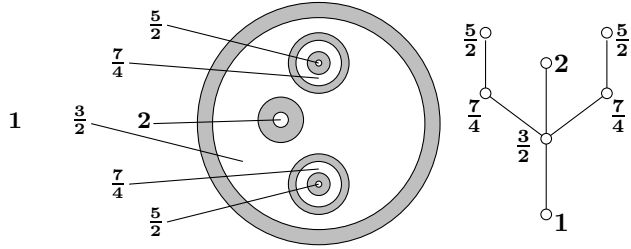
\begin{center}
\begin{tikzpicture}
\draw[thin](3,0)--(3,2);
\draw[thin](3,1)--(2.2,1.6);
\draw[thin](3,1)--(3.8,1.6);
\draw[thin](2.2,2.3)--(2.2,1.6);
\draw[thin](3.8,2.3)--(3.8,1.6);

\draw[fill=white](3,0)circle(2pt); 
\draw[fill=white](3,1)circle(2pt); 
\node(a)at(3.2,-.1){{\bf 1}};
\node(a)at(2.7,.9){{$\bf\frac32$}};
\draw[fill=white](3,2)circle(2pt);
\node(a)at(3.2,2){{$\bf2$}};
\draw[fill=white](2.2,1.6)circle(2pt); \draw[fill=white](3.8,1.6)circle(2pt); 
\node(a)at(3.95,1.3){{$\bf\frac74$}};\node(a)at(2.05,1.3){{$\bf\frac74$}};
\draw[fill=white](2.2,2.4)circle(2pt);\draw[fill=white](3.8,2.4)circle(2pt);
\node(a)at(2,2.3){{$\bf\frac52$}};\node(a)at(4,2.3){{$\bf\frac52$}};

\draw[fill=lightgray](0,1.2)circle(1.6);
\draw[fill=white](0,1.2)circle(1.4);
\draw[fill=lightgray](-.5,1.25)circle(.3);
\draw[fill=white](-.5,1.25)circle(.1); 
\draw[fill=lightgray](0,2)circle(.4);
\draw[fill=white](0,2)circle(.3);
\draw[fill=lightgray](0,2)circle(.15);
\draw[fill=white](0,2)circle(.04);
\draw[fill=lightgray](0,.4)circle(.4);
\draw[fill=white](0,.4)circle(.3);
\draw[fill=lightgray](0,.4)circle(.15);
\draw[fill=white](0,.4)circle(.04);

\draw[thin](-2,2.5)--(0,2);\node(a)at(-2.15,2.5){\small{$\bf\frac52$}};
\draw[thin](-2.5,2)--(0,1.8);\node(a)at(-2.65,2){\small{$\bf\frac74$}};
\draw[thin](-3,1.2)--(-1,.7);\node(a)at(-3.15,1.25){\small{$\bf\frac32$}};
\draw[thin](-2.18,1.2)--(-.5,1.25);\node(a)at(-2.3,1.25){\small{$\bf2$}};

                                   \node(a)at(-4,1.25){\small{$\bf1$}};
\draw[thin](-2.5,0.4)--(0,.6);\node(a)at(-2.65,0.4){\small{$\bf\frac74$}};
\draw[thin](-2,-.1)--(0,.4);\node(a)at(-2.15,-.1){\small{$\bf\frac52$}};

\end{tikzpicture}
\end{center}
  
\end{example}

We now explain how the geometric decomposition of $F_t$  associated with a suitable $\rho$ enables one to compute the contact between two real arcs  $(\delta,0)$ and $(\delta',0)$ in $(\C^2,0)$. 

Note that  $q_o(\delta,\delta')=q_i(\delta,\delta')$ since we are in $\C^2$. If  $\delta$ and $\delta'$ have distinct tangent semi-lines, then $q_i(\delta,\delta')=1$. Assume that $\delta$ and $\delta'$ have the same tangent real semi-line
and that in  suitable  coordinates $(x,y)$ of $\C^2$, $(\delta,0)$ and $(\delta',0)$ have parametrizations of the form $\delta(t) = (t,y(t))$ and  $\delta'(t) = (t,y'(t))$. Let $\rho$ be a sequence of blow-ups such that the strict transforms $\delta^*$ and ${\delta'}^*$ meet  $\rho^{-1}(0)$ at two distinct smooth points of $\rho^{-1}(0) $.  

\begin{lemma} \label{lem:inner} Let $(\nu)$ and $(\nu')$ be the two vertices of $\cal F$ such that $\delta(t)$ and $\delta'(t)$ intersect respectively $F_{t,\nu}$ and $F_{t,\nu'}$. Then $q_i(\delta, \delta') = q_{\nu, \nu'}$, 
where $q_{\nu, \nu'}$ denotes the maximum among minimum of inner rates  along paths from $(\nu)$ to $(\nu')$ in the graph $\cal F$ (in particular, if $(v)=(v')$ then $q_{v,v'}=q_v$). 
\end{lemma}

\begin{proof}  
  Any path $p_t$ between $\delta(t)$ and  $\delta'(t)$ corresponds to a path  from $(\nu)$ to $(\nu')$ in $\cal F$ which describes the sequence of $F_{t,  {\nu''}}$ and intermediate annuli crossed by $p_t$, and the length of $p_t$ is $\Theta(t^q)$ where $q$ is the minimal inner rate among vertices which are on the path. Since $d_i(\delta(t), \delta'(t))$ is the infimum of $length(p_t)$ among all such paths, this implies $q_i(\delta, \delta') \leq q_{\nu, \nu'}$.  Now,  choose a path from $(\nu)$  to $(\nu')$ maximizing $q$, so with $q=q_{\nu, \nu'}$ and then, a   path $p_t$ in $F_t$ from  $\delta(t)$   {to}  $\delta'(t)$ realizing it. The  path $p_t$ has length  $ \Theta(t^{q_{\nu, \nu'}})$. Therefore $q_{\nu, \nu'} \leq  q_i(\delta,\delta')$. 
  \end{proof}

\subsection{The Polar Wedge Lemma and  the geometric decomposition of a normal surface} \label{subsec:geometric decomposition}

\begin{definition} \label{def:Gauss} Let $(X,0) \subset (\C^n,0)$ be a
  complex surface and let ${\mathscr N } \colon NX \to X$ be the Nash
  modification of $X$ (Definition \ref{def:Nash modification}). The
  lifted Gauss map $\widetilde{\lambda} \colon NX \to \grassman(2,\C^n) $ is  the restriction to $NX$ of the projection of $X  \times \grassman(2,\C^n)$ on the second factor. 
\end{definition}

 Let  $\ell \colon\C^n\to \C^2$ be a linear projection such that the restriction $\ell|_X \colon (X,0) \to (\C^2,0)$ is generic. 
Let $\Pi$ and $\Delta$ be the polar and discriminant curves of $\ell|_X$.
 
 \begin{definition}   \label{local bilipschitz constant}  The \emph{local bilipschitz constant of $\ell|_X$} is the map
$K\colon X\setminus \{0\}\to \R\cup\{\infty\}$ defined as follows. It is
infinite on the polar curve and at a point $p\in X\setminus \Pi$ it is
the reciprocal of the shortest length among images of unit vectors in
$T_{p} X$ under the projection $\ell \mid_{T_{p} X} \colon T_{p} X\to \C^2$.
\end{definition}
 
    Let  $\Pi^*$ denote  the
  strict transform of the polar curve $\Pi$ by the Nash modification $\mathscr N $. Set $B_{\epsilon}  = \{ x \in \C^n \colon \norm x_{\C^n} \leq \epsilon\}$. 
 
 \begin{lemma} \label{lem:local bil bound}Given any neighborhood $U$ of\/ ${\Pi}^*\cap
  \mathscr N ^{-1}(B_\epsilon\cap X)$ in ${NX}\cap
  \mathscr N ^{-1}(B_\epsilon\cap X)$, the local bilipschitz constant $K$ of $\ell$
  is bounded on $B_\epsilon \cap (X\setminus \mathscr N (U))$.
\end{lemma}

\begin{proof}  
  Let $\kappa \colon
  \grassman(2,\C^n) \to \R \cup \{\infty\}$ be the map sending $H \in
  \grassman(2,\C^n)$ to the bilipschitz constant of the restriction
  $\ell|_{H}\colon H \to \C^2$.  The map $\kappa \circ \widetilde{\lambda}$ coincides with $K \circ {\mathscr N }$ on $NX  \setminus
  {\mathscr N }^{-1}(0)$ and takes finite values outside ${\Pi}^*$. The map
  $\kappa  \circ \widetilde{\lambda}$ is continuous and therefore bounded on the
  compact set ${\mathscr N }^{-1}(B_{\epsilon}) \setminus U$.
  \end{proof}
  
The \emph{polar wedge lemma} will introduce some particular sets $\mathscr N (U)$ called {\emph{polar wedges}} which will be $D$-pieces. 

Consider the resolution $\rho'_{\ell} \colon Y_{\ell} \to \C^2$ which resolves the base points of the family of projections of generic polar curves  $(\ell(\Pi_{\cal D}))_{\cal D \in \Omega}$ and let $ \rho$ be $\rho'_{\ell}$ composed with some finite sequence of additional blow-ups of points and $T_0$ the resolution graph of $\rho$.  

 Let $\Pi_0$ be a component of the polar curve $\Pi$ of $\ell$ and let  $\Delta_0= \ell (\Pi_0)$   be the corresponding component of the discriminant curve $\Delta = \ell(\Pi)$. Let  $C_i$ be the irreducible component $(\rho'_{\ell})^{-1}(0)$ which intersects $\Delta_0^*$ and let $(u,v)$ be local coordinates in $Y_{\ell}$ centered at $p =  \Delta_0^*\cap(\rho'_{\ell})^{-1}(0)$ such that $v=0$ is the local equation of $C$. 
 
\begin{definition} \label{def:polar rate}
The rate $q_i$   {associated with $C_i$} is called the {\it polar rate} of $\Pi_0$ and $\Delta_0$. 
\end{definition}

\begin{lemma}[(Polar Wedge Lemma, {\cite[Proposition 3.4]{BNP}})]
  \label{polar wedge lemma} For small $\alpha>0$,
    consider the disc $D(\alpha)=\{u \in \C \colon |u|\leq \alpha\}$
    centered at $p=0$ in $C_{i}$ and let $D(\alpha) \times D^2$
    be the total space of the restriction of the  disc bundle $N(C_i)$
    over $D(\alpha)$. Set  $\cal N_{\Delta_0} = \rho (D(\alpha) \times
    D^2)$ and   let  $A_{\Pi_0}(\alpha)$ be the component of
    $\ell^{-1}( {\cal N}_{\Delta_0})$ which contains $\Pi_0$. Then
   \begin{enumerate}
   \item \label{polar2}
     $A_{\Pi_0}(\alpha)$ is a $D(q_i)$-piece, and when $q_i>1$, it is fibered by its intersections with the real surfaces $\{h=t\} \cap X$, where $h\colon \C^n\to \C$ is a generic linear form;
   \item \label{polar1} The strict transform $A_{\Pi_0}(\alpha)^*$ by $\mathscr N $ is a neighborhood of\/ $\Pi_0^*$ in $NX$  which has limit $\Pi_0^*$ as $\alpha\to 0$.
\end{enumerate}
   \end{lemma}

\begin{definition} \label{def:polar wedge} The set
   $$A(\alpha) = \bigcup_{\Pi_0 \subset \Pi} A_{\Pi_0}(\alpha)$$ is
     called a \emph{polar wedge} around $\Pi$,  and its image $\cal N_\Delta(\alpha)$ by $\ell$ is a \emph{$\Delta$-wedge}.
\end{definition} 

Notice that $\cal N_{\Delta_0}$ is a $D(q_i)$-piece inside the $q_i$-piece $B_i = \rho(\cal N(C_i))$. We can refine the geometric decomposition of $\C^2$ associated with $\rho$ by decomposing $B_i$ into the union of  $ {\cal N}_{\Delta_0}$ and $\overline{B_i \setminus {\cal N}_{\Delta_0}}$, which is still a $B(q_i)$-piece. As a consequence of Lemma \ref{lem:local bil bound} and of the Polar Wedge Lemma \ref{polar wedge lemma}, each component $B(q)$ (resp.\ $A(q,q')$) of this geometric decomposition of $\C^2$ lifts to components of the same type $B(q)$ (resp.\ $A(q,q')$) which are fibered  by their intersections with the real surfaces $\{h=t\} \cap X$. We obtain a geometric decomposition of the germ $(X,0)$ as a union of $B$- and 
$A$- pieces. If $\Delta_0$ is a component of $\Delta$, with the same notations as above, we can now amalgamate each $D(q_i)$-piece component of  $\ell^{-1}(\cal N_{\Delta_0}))$  with the adjacent component of $\ell^{-1}(B'_i)$, forming a new $B(q_i)$-piece which is a component of $\ell^{-1}(B_i)$ (see \cite[Lemma 13.1]{BNP})

Summarizing, we then obtain that each piece of the  geometric decomposition  of $\C^2$ associated with $\rho$ lifts to a  union of pieces of the same type and rate, giving a geometric decomposition of $(X,0)$ as a union of $A$- and $B$- pieces. 

\begin{definition}
We call the decomposition $$(X,0) =  \bigcup_{i \in V(T_0)} \ell^{-1}(B_i) \cup   \bigcup_{(i,j) \in E(T_0)} \ell^{-1}(A_{i,j})$$
the \emph{geometric decomposition} of $(X,0)$ associated with $\rho$. 
\end{definition}

Let $\cal G$ be the graph whose vertices are in bijection with the component $B$-pieces of the decomposition and whose edges are in bijection with intermediate $A$-pieces in such a way that the edge associated with an $A$-piece joins the two vertices corresponding to the two $B$-pieces adjacent to it. We weight each vertex by the rate of the corresponding  $B$-piece. We then obtain a natural surjective graph-map $\cal C \colon \cal G \to T_0$. 

\subsection{Inner contacts between real arcs on a normal surface}\label{sec: inner
  contact on a normal surface}

Let $(\delta,0) \subset (\C^2,0)$ be a real arc whose strict transform by $\rho$ intersects $\rho^{-1}(0)$ at a smooth point. We now extend the results of Section  \ref{subsec: geometric decomposition of C^2} to compute the inner contact  in $(X,0)$ between two real arcs components $\delta_1$  and $\delta_2$ of $\ell^{-1}(\delta)$. 

We choose coordinates $\ell\colon (x_1,\dots,x_n) \to(x_1,x_2)$ and so that  $\delta(t) = (t,x_2(t))$. Intersecting the pieces of the geometric decomposition of $(X,0)$ associated with $\rho$ with the real surface ${\widehat F}_t = \{x_1=t\} \cap X$, we get a decomposition of $\widehat {F}_t$. Let $\widehat{\cal F}$ be the graph whose vertices are in bijection with the components of the intersections of the  $B$-pieces, the edges are in bijection with intermediate annuli and the vertices are weighted  by inner rates (we call it the \emph{fiber-graph}).   If $(\nu)$ is a vertex of  $\widehat{\cal F}$, we will denote by $\widehat{F}_{t,\nu}$ the corresponding component of the decomposition of $\widehat {F}_t$.  

The following is an extension of Lemma  \ref{lem:inner}. 

 \begin{lemma}\label{lem:inner X}
Let $(\nu_1)$ and $(\nu_2)$ be the two vertices of $\widehat{\cal F}$ such that $\delta_1(t)$ and $\delta_2(t)$ intersect respectively $\widehat F_{t,\nu_1}$ and $\widehat F_{t,\nu_2}$.
Then $q_i(\delta_1, \delta_2) = q_{\nu_1, \nu_2}$,  where $q_{\nu_1, \nu_2}$ denotes the maximum among minimum of inner rates  along paths from $(\nu_1)$ to $(\nu_2)$ in the fiber-graph $\widehat{\cal F}$
(again with $q_{\nu_1, \nu_2} = q_{\nu_1}$ if $(\nu_1)=(\nu_2)$). 
\end{lemma}
 \begin{proof} The proof is a straightforward extension of that of Lemma  \ref{lem:inner}.   Any path $p_t$ between $\delta_1(t)$ and  $\delta_2(t)$ corresponds to a path  from $(\nu)$ to $(\nu')$ in $\widehat{\cal F}$ which describes the sequence of $\widehat{F}_{t,  {\nu''}}$ and intermediate annuli crossed by $p_t$, and the length of $p_t$ is $\Theta(t^q)$ where $q$ is the minimal inner rate among vertices which are on the path. This implies $q_i(\delta_1, \delta_2) \leq q_{\nu, \nu'}$.  Now, choose a path from $(\nu)$  to $(\nu')$ in $\widehat{\cal F}$ maximizing $q$, so with $q=q_{\nu, \nu'}$ and then choose a   path $p_t$ in $\widehat{F}_t$ from  $\delta_1(t)$   to  $\delta_2(t)$ realizing it.  This path has length  $ \Theta(t^{q_{\nu, \nu'}})$. Therefore $q_{\nu, \nu'} \leq  q_i(\delta_1,\delta_2)$.
 \end{proof}

We also have a natural graph-map respecting inner rates $\widehat{\cal E} \colon \widehat{\cal F} \to \cal G$. 

To summarize, we have constructed a commutative diagram with four graph-covers respecting inner rates:
 
$$\xymatrix{\widehat{\cal F}
  \ar[r]^{ \widehat{\cal E}} \ar[d]_{\widehat{\cal C}}& \cal G \ar[d]^{\cal C}  \\
  \cal F \ar[r]^{\cal E} & T_0}$$

\subsection{The graph $\cal G$ and  the resolution of $(X,0)$.} \label{subsec:cal G-resolution}

We use again the notations  $\rho_\ell$, $T$ and $\pi_0 \colon X_0 \to X$ in Remark  \ref{rem:notations} and Definitions \ref{def:node T} and \ref{def:node G resolution}.
 Consider the pull-back $X_{\ell}$ of $\rho_{\ell}$ and
  $\ell$ and let   $\alpha_{\ell} \colon X'_{\ell}  \to X_{\ell}$ be the minimal  good resolution of $X_{\ell}$. This induces a resolution $\pi_{\ell} \colon X'_{\ell} \to X$   which  factors through $X_0$ and  a projection  $\widetilde{\ell} \colon X'_{\ell} \to Z_{\ell}$. We will denote by $G$ the resolution graph of $\pi_{\ell} $. The situation is summarized in the following  commutative diagram 
 
    $$
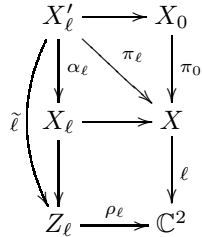
\xymatrix{   X'_{\ell}  \ar@/_1pc/[dd]_{\widetilde{\ell}} \ar[d]^{\alpha_{\ell}}  \ar[r]  \ar[rd]^{\pi_{\ell}} & X_0 \ar[d]^{\pi_0}  \\
    X_{\ell} \ar[r]   \ar[d]  & X \ar[d]^{\ell}  \\
    Z_{\ell}\ar[r]^{\rho_{\ell}} & \C^2    }$$

Let  $(i)$ be a vertex  of $T$ and $C_i$ the corresponding exceptional curve in $\rho_{\ell}^{-1}(0)$. The inverse image of $C_i$ by $\widetilde \ell$  is a union of $k_i$ exceptional curves $E_{i,j}$, $j=1 \ldots, k_i$ of the exceptional divisor of  $\pi_{\ell} $. So $N(C_i)$ lifts by  $\widetilde \ell$ to $k_i$ connected components $N(E_{i,j})$, $j=1 \ldots, k_i$ where $N(E_{i,j})$ is a disc neighborhood of $E_{i,j}$ in $X'_{\ell}$.  Therefore, the inverse image  by $\ell$  of $B_i= \rho_{\ell}(\cal N(C_i))$ consists of $k_i$ components of the geometric decomposition of $(X,0)$. This implies that the graph $G$ is a refinement of $\cal G$ in the sense that we have an inclusion $\cal I: V(\cal G) \to V(G)$ between the sets of vertices  and the graph $\cal G$ is obtained from $G$ by replacing some strings of valency $2$ vertices in  $G$ by edges. More precisely,  the vertices of $\cal G$ correspond to the irreducible  components $E_{\nu}$ of the exceptional divisor of $\pi_{\ell}$ such that $\widetilde{\ell}(E_{\nu})$ is a curve in $\rho_{\ell}^{-1}(0)$.  If $(\nu)$ is a vertex of $\cal G$, we will also denote $(\nu)$ the corresponding vertex $\cal I(\nu)$ in $G$ and $E_{\nu}$ the corresponding component of the exceptional divisor of $\pi_{\ell}$. For more details, see  Sections 13 and 14 in \cite{BNP}.

\begin{example}
In this example we will give the four graphs for the simple surface singularity $E_8$. 

We can assume that $E_8$ has equation $x^2+y^3+z^5=0$ and then the
projection $\ell(x,y,z)=(y,z)$ is generic, so the discriminant is
given by $y^3+z^5=0$. Below is the graph $G$ for $E_8$ on the left and the
resolution graph of the discriminant curve $\Delta$ discriminant on the right. In this example, the graphs $G_0$ and $G$ coincide.  In $G$ the dashed
arrows represent the strict transform of the polar, and the dotted
arrows represent the strict transform of a generic hyperplane
section. In the graph on the right the dashed arrow represents the
strict transform of the discriminant. The circled vertices of $G$ are the vertices corresponding to vertices of $\cal G$ by the injection $\cal I$, and they are weighted by the corresponding inner rates.

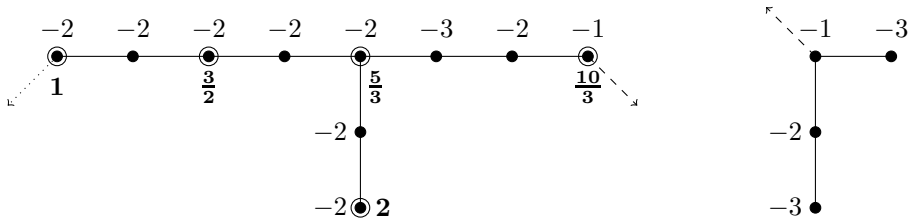
\begin{center}
 \begin{tikzpicture}

\draw[thin](0,0)--(1,0);
\draw[thin](0,0)--(-1,0);
\draw[thin](0,0)--(0,-1);
\draw[thin](1,0)--(2,0);
\draw[thin](2,0)--(3,0);
\draw[thin](-1,0)--(-2,0);
\draw[thin](-2,0)--(-3,0);
\draw[thin](-3,0)--(-4,0);
\draw[thin](0,-1)--(0,-2);

\draw[dashed,>-stealth,->](3,0)--+(.65,-.65);
\draw(3,0)circle(3.5pt);

\draw[dotted,>-stealth,->](-4,0)--+(-.65,-.65);
\draw(-4,0)circle(3.5pt);

\draw[fill=black](0,0)circle(2pt);
\draw(0,0)circle(3.5pt);
\node(a)at(0,.35){{$-2$}};

\draw[fill=black](1,0)circle(2pt);
\node(a)at(1,.35){{$-3$}};

\draw[fill=black](2,0)circle(2pt);
\node(a)at(2,.35){{$-2$}};

\draw[fill=black](3,0)circle(2pt);
\node(a)at(3,.35){{$-1$}};

\draw[fill=black](-1,0)circle(2pt);
\draw(-2,0)circle(3.5pt);
\draw(0,-2)circle(3.5pt);
\node(a)at(-1,.35){{$-2$}};

\draw[fill=black](-2,0)circle(2pt);
\node(a)at(-2,.35){{$-2$}};

\draw[fill=black](-3,0)circle(2pt);
\node(a)at(-3,.35){{$-2$}};

\draw[fill=black](-4,0)circle(2pt);
\node(a)at(-4,.35){{$-2$}};
\node(a)at(-4,-.4){{$\bf 1$}};
\node(a)at(0.2,-.4){{$\bf \frac{5}{3}$}};
\node(a)at(-2,-.4){{$\bf \frac{3}{2}$}};
\node(a)at(3,-.4){{$\bf \frac{10}{3}$}};
\node(a)at(0.3,-2){{$\bf 2$}};

\draw[fill=black](0,-1)circle(2pt);
\node(a)at(-.4,-1){{$-2$}};

\draw[fill=black](0,-2)circle(2pt);
\node(a)at(-.4,-2){{$-2$}};

\draw[thin](6,0)--(7,0);
\draw[thin](6,0)--(6,-1);
\draw[thin](6,-1)--(6,-2);


\draw[fill=black](6,0)circle(2pt);
\node(a)at(6,.35){{$-1$}};

\draw[dashed,>-stealth,->](6,0)--+(-.65,.65);

\draw[fill=black](7,0)circle(2pt);
\node(a)at(7,.35){{$-3$}};

\draw[fill=black](6,-1)circle(2pt);
\node(a)at(5.6,-1){{$-2$}};

\draw[fill=black](6,-2)circle(2pt);
\node(a)at(5.6,-2){{$-3$}};

\end{tikzpicture}
\end{center}

The following picture shows the 
 graphs $T$, $\cal F$, $\cal G$ and $\widehat{\cal F}$ of $E_8$. The
 graph-maps $\cal C$, $\cal E$, $\widehat{\cal C}$ and $\widehat{\cal
   E}$ respect the shape and colour of a vertex, i.e., $\widehat{\cal E}$ of a
 white vertex with thick boundary is the white vertex with thick
 boundary in $\cal G$ and so on. Notice that in this case the graphs
$T$, and $\cal G$ are equal, but $\cal F$ and $\widehat{\cal F}$ differs
from them and from each other. For another example, see  \ref{ex:Minimal Singularity}.  

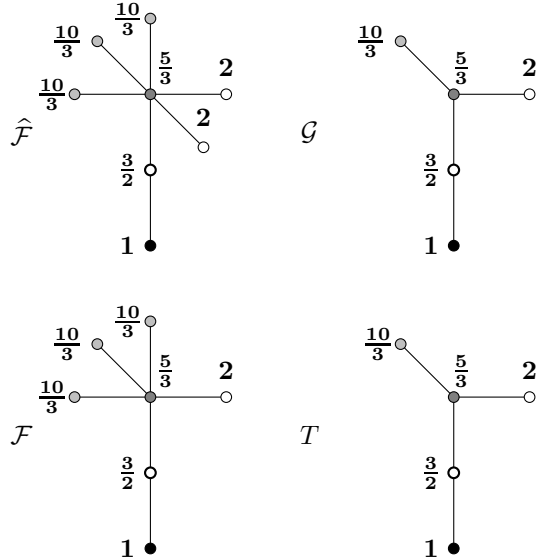
\begin{center}
 \begin{tikzpicture}
 
\draw[thin](0,0)--(1,0);
\draw[thin](0,0)--(0,-1);
\draw[thin](0,0)--(-.7,.7);
\draw[thin](0,-1)--(0,-2);
\draw[thin](0,0)--(-1,0);
\draw[thin](0,0)--(0,1);
\draw[thin](0,0)--(.7,-.7);

\draw[fill=gray](0,0)circle(2pt);
\node(a)at(0.2,.35){{$\bf\frac{5}{3}$}};

\draw[fill=white](1,0)circle(2pt);
\node(a)at(1,.35){{$\bf 2$}};

\draw[fill=white](.7,-.7)circle(2pt);
\node(a)at(.7,-.3){{$\bf 2$}};

\draw[fill=lightgray](-.7,.7)circle(2pt);
\node(a)at(-1.1,.7){{$\bf\frac{10}{3}$}};

\draw[fill=lightgray](-1,0)circle(2pt);
\node(a)at(-1.3,0){{$\bf\frac{10}{3}$}};

\draw[fill=lightgray](0,1)circle(2pt);
\node(a)at(-.3,1){{$\bf\frac{10}{3}$}};

\draw[fill=white,thick](0,-1)circle(2pt);
\node(a)at(-.3,-1){{$\bf\frac{3}{2}$}};

\draw[fill=black](0,-2)circle(2pt);
\node(a)at(-.3,-2){{$\bf 1$}};

\node(a)at(-1.7,-.5){{$\widehat{\cal F}$}};

\draw[thin](4,0)--(5,0);
\draw[thin](4,0)--(4,-1);
\draw[thin](4,0)--(3.3,.7);
\draw[thin](4,-1)--(4,-2);

\draw[fill=gray](4,0)circle(2pt);
\node(a)at(4.1,.35){{$\bf\frac{5}{3}$}};

\draw[fill=white](5,0)circle(2pt);
\node(a)at(5,.35){{$\bf 2$}};

\draw[fill=lightgray](3.3,.7)circle(2pt);
\node(a)at(2.9,.7){{$\bf\frac{10}{3}$}};

\draw[fill=white,thick](4,-1)circle(2pt);
\node(a)at(3.7,-1){{$\bf\frac{3}{2}$}};

\draw[fill=black](4,-2)circle(2pt);
\node(a)at(3.7,-2){{$\bf 1$}};

\node(a)at(2.1,-.5){{$\cal G $}};

\draw[thin](0,-4)--(1,-4);
\draw[thin](0,-4)--(0,-5);
\draw[thin](0,-4)--(-.7,-3.3);
\draw[thin](0,-5)--(0,-6);
\draw[thin](0,-4)--(-1,-4);
\draw[thin](0,-4)--(0,-3);

\draw[fill=gray](0,-4)circle(2pt);
\node(a)at(0.2,-3.65){{$\bf\frac{5}{3}$}};

\draw[fill=white](1,-4)circle(2pt);
\node(a)at(1,-3.65){{$\bf 2$}};

\draw[fill=lightgray](-.7,-3.3)circle(2pt);
\node(a)at(-1.1,-3.3){{$\bf\frac{10}{3}$}};

\draw[fill=lightgray](-1,-4)circle(2pt);
\node(a)at(-1.3,-4){{$\bf\frac{10}{3}$}};

\draw[fill=lightgray](0,-3)circle(2pt);
\node(a)at(-.3,-3){{$\bf\frac{10}{3}$}};

\draw[fill=white,thick](0,-5)circle(2pt);
\node(a)at(-.3,-5){{$\bf\frac{3}{2}$}};

\draw[fill=black](0,-6)circle(2pt);
\node(a)at(-.3,-6){{$\bf 1$}};

\node(a)at(-1.7,-4.5){{$\cal F $}};

\draw[thin](4,-4)--(5,-4);
\draw[thin](4,-4)--(4,-5);
\draw[thin](4,-4)--(3.3,-3.3);
\draw[thin](4,-5)--(4,-6);

\draw[fill=gray](4,-4)circle(2pt);
\node(a)at(4.1,-3.65){{$\bf\frac{5}{3}$}};

\draw[fill=white](5,-4)circle(2pt);
\node(a)at(5,-3.65){{$\bf 2$}};

\draw[fill=black](4,-6)circle(2pt);
\node(a)at(3.7,-6){{$\bf 1$}};

\draw[fill=lightgray](3.3,-3.3)circle(2pt);
\node(a)at(3,-3.3){{$\bf\frac{10}{3}$}};

\draw[fill=white,thick](4,-5)circle(2pt);
\node(a)at(3.7,-5){{$\bf\frac{3}{2}$}};

\node(a)at(2.1,-4.5){{$T$}};

\end{tikzpicture}
\end{center}
\end{example}

\section{Vertically aligned arcs} \label{sec:vertically aligned}

\begin{definition}[(Vertical arc criterion)]\label{def:vertically aligned} Let $(X,0)\subset (\C^n,0)$ be a normal surface. 
Let $\ell \colon (X,0) \to (\C^2,0)$ be a generic projection.  We say {\it pair of vertically aligned real arcs}  for any pair of distinct real arcs $\delta_1,\delta_2$ on $(X,0)$  such that $\ell(\delta_1)=\ell(\delta_2)$. 
\end{definition}

We use again the notations introduced at the beginning of Section \ref{sec:test curve criterion}: let $\rho'_{\ell} \colon Y_{\ell} \to \C^2$ the minimal composition of blow-ups of points  which resolves the base points of the family of projections of generic polar curves  $(\ell(\Pi_{\cal D}))_{\cal D \in \Omega}$. We denote by $\Delta^*$ the strict transform of  the discriminant curve $\Delta$ of $\ell$ by $\rho'_{\ell}$.

 \begin{proposition}\label{liftingoftestarcs} A normal surface $(X,0)\subset (\C^n,0)$ is LNE if and only if for all  generic projections $\ell \colon (X,0) \to (\C^2,0)$, for all real arcs $(\delta,0) \subset (\C^2,0)$ such that $\Delta^* \cap \delta^* = \emptyset$,  any pair of vertically aligned arcs $\delta_1,\delta_2$ in $\ell^{-1}(\delta)$ satisfies the the arc criterion, i.e., $q_i(\delta_1,\delta_2) = q_o(\delta_1,\delta_2)$. 
  \end{proposition}

\begin{proof}
   If $(X,0)$ is normally embedded, then the vertical  arc criterion is a consequence of the arc criterion of Theorem \ref{thm:arc criterion}. 

Conversely, assume that any pair of vertically aligned real arcs by a generic projection satisfies the arc criterion. We have to prove that any pair of real arcs also satisfies the arc criterion.

Let $\delta_1$ and $\delta_2$ be a pair of real arcs on $X$. If their tangent semi-lines are distinct, then  $q_i(\delta_1,\delta_2) = q_o(\delta_1,\delta_2)=1$ and the arc criterion is satisfied. We now assume that $\delta_1$ and $\delta_2$ have the same tangent semi-line. 

We choose a generic projection $\ell \colon (X,0) \to (\C^2,0)$  with polar curve $\Pi$ such that  $ \delta^*_i \cap \Pi^* = \emptyset$ for $i=1,2$, where $^*$ means strict transform by the Nash modification $\mathscr N $.  Then we can choose $\alpha >0$ such that the polar wedge $A(\alpha)$  intersects the germs  $(\delta_1,0)$ and $(\delta_2,0)$ only at $0$.  Consider the two real arcs $\sigma_i = \ell \circ \delta_i \colon [0,\eta)  \to \C^2$. Let $P_t \colon [0,1] \to X$ be the path on $X$ defined as  the lifting by $\ell$ with origin $\delta_1(t)$ of the segment  $S_t$ joining $\sigma_1(t)$ and $\sigma_2(t)$.  Denote by  $ \delta'_2 (t)$ the extremity of the path $P_t$. Then  $(\delta_2, \delta'_2)$ is a pair of vertically aligned  real arcs on $(X,0)$. By hypothesis,  it satisfies the arc criterion so we have $ q_i(\delta'_2,\delta_2) = q_o(\delta'_2,\delta_2)$.

 Assume first that the path $P_t$ does not intersect
 $A(\alpha)$. By Lemma \ref{lem:local bil bound}, we can choose a real $K_0 \geq 1$ such that the local bilipschitz constant of $\ell$ is bounded by $K_0$ on  $B_{\epsilon} \cap (X \setminus A(\alpha))$. 
 
  We then
 have $$\norm{\delta_1(t)- \delta'_2(t)} \leq d_i(\delta_1(t), \delta'_2(t)) \leq \hbox{length}(P_t) \leq K_0\, \hbox{length}(S_t) \leq K_0 \norm{\delta_1(t)- \delta_2(t)}.$$

Therefore, 
  $$\norm{\delta_2(t)- \delta'_2(t)} \leq  \norm{\delta_2(t)- \delta_1(t)} +  \norm{\delta_1(t)- \delta'_2(t)} \leq (1+K_0) \norm{\delta_1(t)- \delta_2(t)}\,.$$
We then obtain 
  $$\norm{ \delta_1(t)- \delta_2(t)} \leq d_i(\delta_1(t),\delta'_2(t)) +  \norm{ \delta'_2(t)- \delta_2(t)} \leq (1+2K_0) \norm{ \delta_1(t)- \delta_2(t)},$$
  which implies 
  $$ q_o(\delta_1,\delta_2) = \min(q_i(\delta_1,\delta'_2), q_o(\delta_2,\delta'_2)).$$

On the other hand, we have $d_i(\delta_1(t), \delta_2(t)) \leq d_i(\delta_1(t),\delta'_2(t)) + d_i(\delta'_2(t), \delta_2(t))$,  therefore
$$\min(q_i(\delta_1,\delta'_2), q_i(\delta_2,\delta'_2)) \leq  q_i(\delta_1,\delta_2).$$
Since $ q_i(\delta'_2,\delta_2) = q_o(\delta'_2,\delta_2)$, we then obtain
$$ q_o(\delta_1,\delta_2)  = \min(q_i(\delta_1,\delta'_2), q_o(\delta_2,\delta'_2)) \leq q_i(\delta_1,\delta_2),$$
which implies $q_o(\delta_1,\delta_2) =q_i(\delta_1,\delta_2)$ since  $q_i(\delta_1,\delta_2) \leq q_o(\delta_1,\delta_2)$. Therefore, the pair $\delta_1,\delta_2$ satisfies the arc criterion for normal embedding. 

Assume now that the path $P_t$ intersects the polar-wedge $A(\alpha)$. 

 Let us treat first the case for $\alpha>0$ sufficiently small. Let $A'(\alpha)$ denote the union of components of $A(\alpha)$ which intersect $P_t$. {If $t$ is sufficiently small,} we have $\delta_1(t) \not\in A'(\alpha)$ by choice of $\ell$,  so $\sigma_1(t) \not\in {\ell(A'(\alpha))}$  and also   $\delta'_2(t) \not\in A'(\alpha)$, so $\sigma_2(t) \not\in  \ell(A'(\alpha))$.  Then we can replace the segment $S_t$ by a path $S_t$ from $\sigma_1(t)$ to $\sigma_2(t)$ such that $\hbox{length}(S'_t) \leq \pi \hbox{length} (S_t)$ and such that the lifting of $S'_t$ by $\ell$ with extremities $\delta_1(t)$ and $\delta'_2(t)$ does not intersect $A(\alpha)$. Then the previous inequalities are modified by a factor of $\pi$ and we obtain:
   $$\norm{ \delta_1(t)- \delta_2(t)}\leq d_i(\delta_1(t),\delta'_2(t)) + \norm{\delta'_2(t)- \delta_2(t)} \leq (1+2K_0 \pi )\norm{ \delta_1(t)- \delta_2(t)},$$
    which leads to $q_o(\delta_1,\delta_2) = q_i(\delta_1,\delta_2)$ by the same arguments as before. 
    
Let us now assume that for all $\alpha >0$, $\delta'_2 \cap A(\alpha) \neq \{0\}$.  Let $A_0(\alpha)$ be the component of $A(\alpha)$ containing $\delta'_2$.  Let $\alpha' >\alpha$ and let $\widetilde{\sigma}_2$ be a real arc  inside  $\ell(A_0(\alpha')) $ such that for all $t$,  $\widetilde{\sigma}_2(t) \in S_t$ and  the strict transform of  $\widetilde{\sigma}_2(t)$ by $\rho'_{\ell}$ does not intersect $\Delta^*$. Then decreasing   $\alpha$ if necessary, we can assume   $\widetilde{\sigma}_2(t) \not\in \cal \ell(A_0(\alpha))$. 
    Consider the path $\beta'_t$ defined as the lifting by $\ell$ with origin $\delta_2(t)$ of the segment $\beta_t = [\widetilde{\sigma}_2(t), \sigma_2(t)]$. Let   $ \widetilde{\delta}_2(t)$ be the extremity of $\beta'_t$ and let $\widetilde{\delta}_2'(t)$   be the point of $P_t$ such that  $\ell(\widetilde{\delta}_2'(t)) = \widetilde{\sigma}_2(t)$. Then $\widetilde{\delta}_2'(t)$ and $\widetilde{\delta}_2(t)$ are vertically aligned. Since both  $\widetilde{\delta}_2'(t)$ and $\widetilde{\delta}_2(t)$ are outside $A(\alpha)$, we can apply what we just proved and we obtain $q_i(\delta_1, \widetilde{\delta}_2) = q_o(\delta_1, \widetilde{\delta}_2)$. Call $q$ this number and let $s$ the polar rate of    $A_0(\alpha)$ (Definition \ref{def:polar rate}). 
  
 Since $\delta_1$ and $\delta_2$ have the same tangent semi-line, then $\sigma_2$ and $\widetilde{\sigma}_2$ also have  the same tangent semi-line $L$. Therefore, we have  $d_o(\sigma_2(t),\widetilde{\sigma}_2(t) ) =  \hbox{length}(\beta_t)= \Theta(t^s)$.

We then obtain: 
    $$d_i(\widetilde{\delta}_2, \delta_2)  \leq  \hbox{length}(\beta'_t) \leq K_0\,   \hbox{length}(\beta_t)   =   \Theta(t^s),$$
which implies: 
    $$d_i(\delta_1, \delta_2)  \leq  d_i(\delta_1,\widetilde{\delta}_2)  + d_i( \widetilde{\delta}_2, \delta_2 )  =    \Theta(t^q) + \Theta(t^s).$$ 
    On the other hand, we have 
    $$ \Theta(t^q) + \Theta(t^s) =d_o(\sigma_1, \sigma_2) \leq d_o(\delta_1, \delta_2).  $$
    Since  $d_o(\delta_1, \delta_2) \leq d_i(\delta_1, \delta_2)$, we obtain $d_o(\delta_1, \delta_2) =  d_i(\delta_1, \delta_2) = \Theta(t^q) + \Theta(t^s)$, and then  $q_i(\delta_1, \delta_2)  = q_o(\delta_1, \delta_2)$.
    \end{proof}
\begin{remark}\label{rem:one projection} We proved Proposition \ref{liftingoftestarcs} because it is what we need to prove Proposition \ref{prop:characterization of normal embedding2}. But with a little more work one can adapt the proof  to get a criterion using just one fixed  generic projection $\ell \colon (X,0) \to (\C^2,0)$, if one consider all real arcs $\delta$ of $(\C^2,0)$, even those such that $\Delta^* \cap \delta^* \neq \emptyset$.
It is worth noting that the statement of Proposition \ref{liftingoftestarcs} can also be easily improved by reducing the criterion to pairs of vertically aligned real arcs corresponding to three fixed  generic projections $\ell_1, \ell_2, \ell_3 \colon (X,0) \to (\C^2,0)$. Indeed, if $\ell_1$ is chosen, it suffices to chose $\ell_2$ and $\ell_3$ so that $(\ell_1(\Pi_2))^* \cap \Delta_1^* = \emptyset$, $(\ell_1(\Pi_3))^* \cap \Delta_1^* = \emptyset$ and $(\ell_1(\Pi_2))^* \cap (\ell_1(\Pi_3))^* = \emptyset$ where $^*$ means strict transform by $\rho'_{\ell_1}$  and where for $i=1,2,3$,  $\Pi_i$ is the polar curve of $\ell_i$ and $\Delta_i = \ell_i(\Pi_i)$. Then, for any pair of real arcs $\delta_1$ and $\delta_2$ on $(X,0)$, at least one of $\ell_1$, $\ell_2$ or $\ell_3 $ satisfies $\Delta_i^*  \cap \ell_i(\delta_1)^* = \emptyset$ and $\Delta_i^*  \cap \ell_i(\delta_2)^* = \emptyset$.  
\end{remark}

\section{Partner pairs} \label{sec:partner pairs}

Let $(X,0)$ be a normal surface germ and let $\ell \colon (X,0) \to (\C^2,0)$ be a generic projection.  Let $(\delta,0)$ and $(\delta',0)$ be two real arcs in $(\C^2,0)$ which  have parametrizations of the form $\delta(t)  = (t,y(t))$ and $\delta'(t)  = (t,y'(t))$ in suitable coordinates and which meet the discriminant curve $\Delta$ of $\ell$ only at $0$. Let  $\delta_1, \delta_2$ be a pair of components of the lifting $\ell^{-1}(\delta)$. Let   $S_t$ be the segment in $F_t$ joining $\delta(t)$ and $\delta'(t)$. Let  $P_{1,t}$ and $P_{2,t}$ be the liftings of  $S_t$   by  the cover $\ell {\mid}_{\widehat{F}_t} \colon \widehat{F}_t \to {F}_t$ with origins respectively $\delta_1(t)$ and $\delta_2(t)$.  Denote their extremities by  $\delta'_1(t)$ and $\delta'_2(t)$. This defines a pair $\delta'_1, \delta'_2$ of distinct components of $\ell^{-1}(\delta')$.

\begin{definition} \label{def:partner pairs}
We say that $\delta'_1, \delta'_2$ is the {\it partner pair} of $\delta_1, \delta_2$ over $\delta'$.
\end{definition}

\begin{lemma} \label{lem:partner} Assume that   $\delta_1, \delta_2$ has a partner pair $\delta'_1, \delta'_2$ such that $q_i(\delta'_1, \delta'_2) < q_i(\delta, \delta')$.  If the pair  $\delta_1, \delta_2$  satisfies  the arc criterion $q_i(\delta_1, \delta_2)=q_o(\delta_1, \delta_2)$, then  $\delta'_1, \delta'_2$ also satisfies it.
\end{lemma}
 \begin{proof} 
  We have  $d_i(\delta'_1(t),\delta'_2(t)) \leq length (P_{1,t}) + d_i(\delta_1(t),\delta_2(t)) +  length (P_{2,t})$.  
  
  Since  $q_i(\delta'_1, \delta'_2)  < q_i(\delta, \delta')$ and since  for $k=1,2$, we have $length (P_{k,t}) = \Theta(t^{q_i(\delta, \delta')})$,  we then obtain $q_i(\delta'_1,\delta'_2) \geq  q_i(\delta_1, \delta_2)$. 
  
  We also have  $d_i(\delta_1(t),\delta_2(t)) \leq length (P_{1,t}) + d_i(\delta'_1(t),\delta'_2(t)) +  length (P_{2,t})$, which leads to $q_i(\delta'_1,\delta'_2) \leq  q_i(\delta_1, \delta_2)$ with the same argument.  Therefore  $q_i(\delta'_1,\delta'_2) =  q_i(\delta_1, \delta_2)$. 

 The inequality $d_o(\delta_1(t),\delta_2(t)) \leq  length (P_{1,t}) + d_o(\delta'_1(t),\delta'_2(t)) +  length (P_{2,t})$ gives $q_o(\delta_1,\delta_2) \geq q_o(\delta'_1,\delta'_2)$ and  
 $d_o(\delta'_1(t),\delta'_2(t)) \leq  length (P_{1,t}) + d_o(\delta_1(t),\delta_2(t)) +  length (P_{2,t})$  leads to $q_o(\delta'_1,\delta'_2) \geq q_o(\delta_1,\delta_2)$. Therefore  $q_o(\delta'_1,\delta'_2) =  q_o (\delta_1, \delta_2)$. 
 
 Since $q_o (\delta_1, \delta_2)=q_i (\delta_1, \delta_2)$, then  $q_o (\delta'_1, \delta'_2)=q_i (\delta'_1, \delta'_2)$ as desired. 
 \end{proof}

 \begin{lemma} \label{lem:partner2}     Let $\rho \colon Y \to \C^2$ be a sequence of blow-ups of points which
resolves the base points of the family of projected polar curves
$(\ell(\Pi_{\cal D}))_{\cal D\in \Omega}$  and let  $T_0$ be the resolution graph of $\rho$.  Let $\delta$  a  real arc in $(\C^2,0)$ whose strict transform intersects $\rho^{-1}(0)$ at a smooth point $p$  and  let $\delta_1, \delta_2$ be a pair of vertically aligned arcs over $\delta$.  Let  $\delta'$ be another real arc in $(\C^2,0)$ and let $\delta'_1, \delta'_2$ be the partner pair of $\delta_1, \delta_2$ over $\delta'$.  {Let $\Delta$ be the discriminant curve of $\ell$.} We assume that none of $\delta^*$ and ${\delta'}^*$ intersects $\Delta^*$ and that the pair  $\delta_1, \delta_2$  satisfies  the arc criterion. 
 \begin{enumerate}
\item  \label{partner1}  If  ${\delta'}^*$ passes through $p$,  then  $\delta'_1, \delta'_2$  also satisfies the arc criterion;
\item  \label{partner2} Let {$C_j$ be the component of $\rho^{-1}(0)$ such that $p \in C_{j}$} and let $(\nu_1)$ and $(\nu_2)$ be the two vertices of $\widehat{\cal F}$  such that $\delta_1(t) {\in} \widehat{F}_{t,\nu_1}$ and $\delta_2(t) {\in} \widehat{F}_{t,\nu_2}$, {so we have $(\cal E \circ \widehat{\cal C})(\nu_1)=(\cal E \circ \widehat{\cal C})(\nu_2)=(j)$}.  Let $(\nu_0)$ be a vertex of $\widehat{\cal F}$ on a path from $(\nu_1)$ and $(\nu_2)$ such that $q_{\nu_1,\nu_2} = q_{\nu_0}$. Assume $q_{\nu_0} < q_{j} $  and set ${(j_0)} = ( \cal E \circ \widehat{\cal C})(\nu_0)$. Let $T_0' \subset T_0$ be the connected component of $T_0 \setminus (j_0)$ which contains $(j)$ and set $E'=\bigcup_{k \in V(T_0')}C_k$. If ${\delta'}^* \cap E'\neq \emptyset$, then $\delta'_1, \delta'_2$  also satisfies the arc criterion.
\end{enumerate}
\end{lemma}

 \begin{proof}   \ref{partner1}  Let {$C_j$ be the component of $\rho^{-1}(0)$ such that $p \in C_{j}$}. Since $\delta$ and $\delta'$ both pass through $p$, then $q_i(\delta, \delta')>q_{j}$. By Lemma \ref{lem:inner X}, we have $q_i(\delta'_1, \delta'_2) =  q_{j}$. Therefore $ q_i(\delta, \delta')> q_i(\delta'_1, \delta'_2)  $ and we then get   \ref{partner1}  by applying Lemma \ref{lem:partner}. 
 
\ref{partner2} By composing $\rho$ with an additional sequence of
blow-ups of points, we can assume that ${\delta'}^*$ intersects
$\rho^{-1}(0)$ at a smooth point. Let {$C_{j'}$} be the component of
$E'$ such that ${\delta'}^* \cap C_{j'} \neq
\emptyset$. Then we have $q_{j_0}<  q_{j'}$. By assumption we also have
$q_{j_0}<  q_{j}$.  Therefore, by Lemma \ref{lem:inner}, we get
$q_i(\delta, \delta') =  q_{j, j'}  > q_{j_0}$ since $(j)$ and
$(j')$ are in the same connected component of $T_0\setminus
(j_0)$. On the other hand,  by Lemma \ref{lem:inner
      X},  $q_i(\delta_1,\delta_2)  = q_{\nu_1, \nu_2}  =q_{\nu_0} {=q_{j_0}}$. We
  then have  $q_i(\delta, \delta') > q_i(\delta_1,\delta_2)$ and we conclude again by applying Lemma \ref{lem:partner}.
\end{proof} 

\section{LNE  along strings} \label{sec:strings}
Let $\ell \colon (X,0) \to (\C^2,0)$ be a generic projection and let
$\rho \colon Y \to \C^2$ be a sequence of blow-ups of points which
resolves the base points of the family of projected polar curves
$(\ell(\Pi_{\cal D}))_{\cal D\in \Omega}$.  Let $E'$ be the union of
components of $ \rho^{-1}(0)$ which are not $\Delta$-curves. Let 
$(\delta,0) \subset (\C^2,0)$ be  a real arc such that $\delta^* \cap E' \neq \emptyset$  and such that $\delta^*$ intersects $\rho^{-1}(0)$ at a smooth point.

\begin{lemma} \label{LNE along strings}  Let $C$ be the component of $ \rho^{-1}(0)$ such that
$C \cap \delta^* \neq \emptyset$ and let $q_C$ be its inner rate.   Let  $\delta_1$ and $\delta_2$ be two real arc components of
$\ell^{-1}(\delta)$ and consider the two points $p_1 = \delta_1^*\cap {\mathscr N }^{-1}(0)$ and  $p_2 = \delta_2^*\cap {\mathscr N }^{-1}(0)$ {where $^*$ means strict transform by the Nash modification $\mathscr N$}. Assume that $q_i(\delta_1,\delta_2) = q_C$.   

Then the pair of arcs  $(\delta_1,\delta_2)$ satisfies the  arc
criterion $q_i(\delta_1,\delta_2) = q_o(\delta_1,\delta_2)$  if and
only if $ \widetilde{\lambda}(p_1) \neq  \widetilde{\lambda}(p_2)$,
where $ \widetilde{\lambda}$ denotes the lifted Gauss map (Definition \ref{def:Gauss}).
\end{lemma}

\begin{proof}  By {\ref{partner1}} of Lemma \ref{lem:partner2}, it suffices to prove the result when $\delta$ is the real slice of a complex curve $\gamma$  whose strict transform by $\rho_{\ell}$ is a  curvette of  $C$. 
We may assume that our coordinates are  $(x,y)$ and that $\delta$ is  parametrized by $\delta(t) = (t,y(t))$. By Lemma \ref{lem:inner X}, the assumption $q_i(\delta_1,\delta_2) = q_C$  is equivalent to asking that in the fiber-graph $\widehat{\cal F}$, the two vertices $(\nu_1)$ and $(\nu_2)$ such that ${\delta_i(t) \in} \widehat{F}_{t,\nu_i}$ are joined by a path along which inner rates are $\geq q_{C}$.

Set $q=q_C$. Then  $\gamma$ has a Puiseux expansion of the form
  $$ y=\sum_{i=1}^k a_i x^{p_i} +  a x^q (1+ x^{q''} b(x)),$$  with $a, a_i \in \C^*$, 
  $1 \leq p_1 < p_2 < \cdots < p_k <q$, $1<q''$ and where the higher
  order terms $a x^{q+q''}b(x) \in \C\{x^{1/n}\}$ contain only non
  essential exponents. Let $q_1$ and $q_2$ be two rational numbers
  such that $q_1<q<q_2$ and such that any branch of the discriminant
  curve $\Delta$ of $\ell$ with Puiseux expansion of the form
  $ y=\sum_{i=1}^k a_i x^{p_i} + h.o.$, where $h.o.$ means higher
  order terms, satisfies the following property: the first exponent
  $>p_k$ is not inside the interval $[q_1, q_2]$.  Set
  $\alpha(x)=\sum_{i=1}^k a_i x^{p_i}$. This is equivalent to asking that
  the strict transforms by $\rho_{\ell}$ of the curves with Puiseux
  expansion $ y=\alpha(x) + x^{q'}$ with $q' \in [q_1,q_2]$ intersects
  $\rho_{\ell}^{-1}(0)$ along a union of curves
  $C_{u_1} \cup \ldots \cup C_{u_r} $ in $E'$ which corresponds to a
  string $(u_1)-(u_2)- \cdots -(u_r)$ in the resolution tree of
  $\rho_{\ell}$. Then $\delta$ is contained in the real
  $3$-dimensional semialgebraic germ $\cal A$ with boundary defined
  by:
$$\cal A  = \{(t,y) \in \R^+ \times \C \colon   t^{q_2} \leq |y- \alpha(t)| \leq   t^{q_1}  \}\,.$$
If $n$ is the multiplicity of $\gamma$ then $\cal A$ is the union of $n$ real $3$-dimensional semialgebraic germs with boundary which are pairwise disjoint outside $0$. We write $(A,0)$ for the  one containing $\delta$.

Let $m$ be the multiplicity of $(X,0)$. Then  $\ell^{-1}(A)$ consists
of $m$ semialgebraic germs $(A^{(i)},0)$, $i=1,\ldots,m$, which are
pairwise disjoint outside $0$. For each $i$,  the restriction $\ell
\colon A^{(i)} \to A$ is an inner bilipschitz homeomorphism  by the Polar Wedge Lemma  \ref{polar wedge lemma}.  Let $\ell'_i \colon (t,y) \mapsto  \ell'_i(t,y) \in A^{(i)}$ be  the inverse map   of $\ell \mid_{A^{(i)}}$.
Assume $\delta_1 \subset A^{(1)}$ and $\delta_2 \subset A^{(2)}$.

 For  $s \in \C$ such that $ t^{q_2} \leq |s| \leq   t^{q_1} $, consider the function $g \colon A \to (0, + \infty)$ defined by
$$g(t,s)=  \norm{\ell'_1(t,\alpha(t)+ a s) - \ell'_2(t,\alpha(t)+  a s
)}.$$
 Set $s_0=t^{q}(1+t^{q''}b(t))$. We have in particular: $$ g(t, s_0)= \norm{\delta_1(t) - \delta_2(t)}$$

Let us give an estimate of  $g(t, s_0) $.  Fix a small
$t \in \R^+$ and let us write the Taylor formula for
$s \mapsto g(t,s)$ at the point $s_0$. We get  for $s \in \C$
  such that $ t^{q_2} \leq |s| \leq t^{q_1} $:
$$g(t, {s_0})- g(t,s) = ({s_0}-s) a_1(t) + ({s_0}-s)^2 a_2(t) +  ({s_0}-s)^3 a_3(t) + \cdots,$$
where $a_1(t)$ equals the distance between the lines
$L^{(1)}_t = T_{ {\delta_1(t)}} X \cap \{x=t, y \in \R\}$ and
$L^{(2)}_t = T_{ {\delta_2(t)}} \cap \{x=t, y \in  {\R}\}$ in the
Grassmanian $\grassman(1, \R^{2n})$.

Set $s=t^{q'}$ where $q<q'\leq q_2$.  Then $s_0-t^{q'} = t^{q}(1+t^{q''}b(t)) - t^{q'} = \Theta(t^{q})$. We then have: 
$$  g(t,{s_0}) -   g(t,t^{q'}) = { \Theta(t^q)} a_1(t) + { \Theta(t^{2q})}  a_2(t) +  { \Theta(t^{3q})}  a_3(t) + \ldots, \hskip1cm (1)$$

 By Lemma \ref{lem:inner X}, we have $  q_i(\delta_1 , \delta_2)    = q$  and
 $$d_i(\ell'_1(t,\alpha(t)+at^{q'}), \ell'_2(t,\alpha(t)+at^{q'})) = \Theta(t^{q'})\,.$$
  
The latter equality implies: 
\begin{align*}  
 g(t,t^{q'}) &=  \norm{\ell'_1(t,\alpha(t)+a t^{q'}) -
   \ell'_2(t,\alpha(t)+ a t^{q'}}\\ &\leq
 d_i(\ell'_1(t,\alpha(t)+at^{q'}), \ell'_2(t,\alpha(t)+at^{q'})) =
 \Theta(t^{q'})
\end{align*}  
 
  Then dividing the equality (1) by
  $d_i(\delta_1(t) - \delta_2(t)) = \Theta(t^q)$, we get:
$$  \frac{\norm{\delta_1(t) - \delta_2(t)}}{ d_i(\delta_1(t) - \delta_2(t))} =  \Theta(a_1(t)).$$

Set $P_1 = \widetilde{\lambda}(p_1)$   and $P_2 = \widetilde{\lambda}(p_2)$.   As $t$ tends to zero, $a_1(t)$ tends to the distance between the two
real lines $L_1 = P_1 \cap \{x=t, y \in \R\}$ and
$L_2 = P_2 \cap \{x=t, y \in \R\}$.  If $P_1 = P_2$, we then have
$L_1 = L_2$ so $\lim_{t \to 0} a_1(t) = 0$ and we get
$q_i(\delta_1, \delta_2) > q_0(\delta_1, \delta_2)$.

Since the components of $\ell^{-1}(A)$ which contain $A^{(1)}$ and
$A^{(2)}$ are tangent to the same line $L$, then by Whitney's Lemma
\cite[Theorem 22.1]{W}, we have $L \subset P_1$ and $L \subset
P_2$. Therefore, if $P_1 \neq P_2$, then we also have $L_1 \neq L_2$,
which means that $\lim_{t \to 0} a_1(t) \neq 0$. In that case, we then
have $q_i(\delta_1, \delta_2) = q_0(\delta_1, \delta_2)$.
 \end{proof}

\section{Proof of Proposition \ref{prop:characterization of normal embedding2}} \label{sec:proof}

\begin{definition}  A {\it $\cal P$-node}  of  the fiber-graph $\widehat{\cal F}$ is a vertex $(\nu)$ such that  the image of $\widehat{\cal E}(\nu)$  by the injection $\cal I \colon V(\cal G) \to V(G)$  is a $\cal P$-node of the resolution graph $G$ of $\pi_{\ell}$ (see Section \ref{subsec:cal G-resolution}). 
 Equivalently,  $(\nu)$ is a $\cal P$-node of $\widehat{\cal F}$ if and only if  for all  $t \in (0,\eta)$, ${\Pi} \cap  \widehat{F}_{\nu, t}\neq \emptyset$, where $\Pi$ denotes the polar curve of a generic projection of $(X,0)$.  
 \end{definition} 
 
 The proof will use the relation between the $\cal P$-nodes of
 $\widehat{\cal F}$ and the lifted Gauss map. Let us explain this first.
 
 Let $(\nu)$ be a vertex of  $\widehat {\cal F}$  and let $\delta$ be a real arc {on $(X,0)$} such that for all  $t \in (0,\eta)$, $\delta(t) \in   \widehat{F}_{\nu, t}$. If $(\nu)$ is not a $\cal P$-node, then the intersection point of the strict transform $\delta^*$ of $\delta$ by the Nash modification of $(X,0)$ is a point $p$ which does not depend on the choice of $\delta$. We set  $\widetilde{\lambda}(\nu):=  \widetilde{\lambda}(p)$, where $\widetilde{\lambda}$ is the lifted Gauss map.
 
 Let now $(\nu_1)$ and $(\nu_2)$ be two vertices of $\widehat {\cal F}$ which are in the same connected component of $\widehat {\cal F}$ minus its $\cal P$-nodes and let $\delta_1$ and $\delta_2$ be two real arcs  such that   for all  $t \in (0,\eta)$, $\delta_1(t) \in  \widehat{F}_{\nu_1, t}$ and $\delta_2(t) \in  \widehat{F}_{\nu_2, t}$. Then the strict transforms  of $\delta_1$ and $\delta_2$ by the Nash modification $\mathscr N$ of $(X,0)$ intersect the exceptional divisor at the same point $p$, and we then have $\widetilde{\lambda}(\nu_1) = \widetilde{\lambda}(\nu_2)$. 
   
 \begin{proof}   The ``only if'' direction of Theorem \ref{prop:characterization of normal
  embedding2} is a direct consequence of Theorem \ref{thm:arc  criterion}. 
  
  Let us prove the ``if'' direction. We assume that for all generic projections $\ell \colon (X,0) \to (\C^2,0)$, any pair of   components $\delta_1,\delta_2$ over any test arc $\delta$ of $\ell$  satisfies  the arc criterion $q_i(\delta_1,\delta_2) = q_o(\delta_1,\delta_2)$.  

  Let $\delta \colon [0,\eta) \to \C^2$ be a real arc such that $\Delta^* \cap \delta^* = \emptyset$ where $^*$ means strict transform by $\rho_{\ell}$, and let $\delta_1$ and $\delta_2$ be two components of $\ell^{-1}(\delta)$.  We assume that $\delta$ is not a test arc. By Proposition \ref{liftingoftestarcs}, we have to prove that  $q_i(\delta_1,\delta_2) = q_o(\delta_1,\delta_2)$. 

 Assume  first that $\delta^*$ intersects  the component {$C_{j}$} of $(\rho_{\ell})^{-1}(0)$ at a smooth point $p$ of $(\rho_{\ell})^{-1}(0)$. Let $\delta'$ be a test arc passing through $p$ {and let $\delta'_1,\delta'_2$ be the partner pair of $\delta_1,\delta_2$ over $\delta'$}. Since $\delta'$ is a test arc, the pair $\delta'_1,\delta'_2$  satisfies  the arc criterion, then applying \ref{partner1} of Lemma \ref{lem:partner2},   we obtain that the pair $\delta_1,\delta_2$ also satisfies it. 
 
 Assume now that $\delta^*$ intersects $(\rho_{\ell})^{-1}(0)$ at an intersection point {$p = C_{j'} \cap C_{j''}$}.   Let us compose $\rho_{\ell}$ with an additional sequence of blow-ups of points $\alpha$ so that {${\delta}^*$} intersects $(\rho_{\ell} \circ \alpha)^{-1}(0)$ at a smooth point and let $C_{j}$ be the component of $(\rho_{\ell} \circ \alpha)^{-1}(0)$ which intersects $\delta^*$. In the tree $T$, this replaces the edge between $(j')$ and $(j'')$ by   a string $\cal S = (j') - \cdots - (j ) - \cdots - (j'')$.   We assume $q_{j'}< q_{j''}$.   Let $\delta'$ and $\delta''$
 be  test arcs at $(j')$ resp.\ $(j'')$. Let $\delta'_1, \delta'_2$  be the partner pair of $\delta_1,\delta_2$ over $\delta'$, and the same for $\delta_1'',\delta_2''$ over $\delta''$. For $i=1,2$, let $(\nu_i)$, $(\nu'_i)$ and $(\nu''_i)$ be the vertices of $\widehat{\cal F}$ such that $\delta_i(t) \in \widehat{F}_{t,\nu_i}$,  $\delta'_i(t) \in \widehat{F}_{t,\nu'_i}$ and $\delta''_i(t) \in \widehat{F}_{t,\nu''_i}$. 
 
The string $\cal S$ lifts to two strings $\cal S_1$ and $\cal S_2$ in $\widehat{\cal F}$ with extremities respectively $(\nu'_1)$ and $(\nu''_1)$ and $(\nu'_2)$ and $(\nu''_2)$.  
\vskip0,2cm\noindent
{\bf Case 1.} Assume that $(\nu''_1)$ and $ (\nu''_2)$ can be joined by a path $p$ along which inner rates are $\geq q_{j''}$. We then have $q_{\nu''_1, \nu''_2} = q_{j''}$ and also  $q_{\nu'_1, \nu'_2} = q_{j'}$ since one obtains a simple path from $(\nu'_1)$ to $(\nu'_2)$ by appending $\cal S_1$ and $\cal S_2$ to $p$. Since there are no adjacent $\Delta$-nodes in the tree $T$, one of $(j')$ or $(j'')$, say $(j'')$,  is not a $\Delta$-node (the arguments will be the same when $(j'')$ is a $\Delta$-node and not $(j')$). Since $q_{\nu''_1, \nu''_2} = q_{j''}$ and since the pair  $(\delta''_1, \delta''_2)$ satisfies the arc criterion, then  Lemma \ref{LNE along strings}  implies that   $\widetilde{\lambda}(\nu''_1) \neq \widetilde{\lambda}(\nu''_2)$. Since the vertices $(\nu_i)$ and $(\nu''_i)$ are in the same connected component of  $\widehat{\cal F}$ minus its $\cal P$-nodes, then we have  $\widetilde{\lambda}(\nu'_1) = \widetilde{\lambda}(\nu_1)$ and $\widetilde{\lambda}(\nu'_2) = \widetilde{\lambda}(\nu_2)$. Therefore, $\widetilde{\lambda}(\nu_1) \neq  \widetilde{\lambda}(\nu_2)$. Applying again Lemma \ref{LNE along strings} in the converse direction, we obtain that  $\delta_1, \delta_2$ satisfies the arc criterion.
 
 \vskip0,2cm\noindent
{\bf Case 2.} Assume that $(\nu''_1)$ and $(\nu''_2)$  cannot be joined by a path $p$ along which inner rates are $\geq q_{j''}$. Let $(\nu_0)$ be a vertex of $\widehat{\cal F}$ on a simple path from $(\nu'_1)$ to $(\nu'_2)$ such that  $q_{\nu'_1, \nu'_2}  = q_{\nu_0}$ and set $(j_0) = (\cal E \circ \widehat{\cal C})(\nu_0)$.  Since any   path from $(\nu''_1)$ to $(\nu''_2)$ is obtained by appending the strings $\cal S_1$ and $\cal S_2$ to a path from $(\nu_1)$ to $(\nu_2)$, we have $q_{\nu''_1,\nu''_2} = q_{\nu'_1, \nu'_2} = q_{j_0} < j''$. Then we can apply \ref{partner2} of Lemma \ref{lem:partner2} to the vertex $(j'')$:   since the pair $(\delta''_1, \delta''_2)$ satisfies the arc criterion then the pair $(\delta_1, \delta_2)$ also satisfies it. 
  \end{proof}

\section{Enhanced Proposition  \ref{prop:characterization of normal embedding2}}

\begin{definition}  \label{def:nodal-principal} Let $\ell \colon (X,0) \colon (\C^2,0)$ be a generic projection and let $\delta$ be a real test arc for $\ell$. A component $\widehat{\delta}$ of $\ell^{-1}(\delta)$ is \emph{principal} if it is a real slice of a principal component (Definition \ref{def:G'}) of $\ell^{-1}(\gamma)$. A {\it nodal test arc} is a test arc $\delta$ which is a real slice of a nodal test curve (Definition \ref{def:test curve}). 
  \end{definition}
 
  \begin{proposition}[(Enhanced Proposition \ref{prop:characterization of normal embedding2})]  \label{cor:characterization of normal embedding} A normal surface  $(X,0)$  is LNE if and only if  for all  generic projections $\ell \colon (X,0) \to (\C^2,0)$ and for all nodal test arcs  $\delta$ for $\ell$, any pair of   principal components $\delta_1,\delta_2$ of $\ell^{-1}(\delta)$ satisfies  the arc criterion $q_i(\delta_1,\delta_2) = q_o(\delta_1,\delta_2)$.  
\end{proposition}
 
 \begin{definition}  An {\it $\cal L$-node}  of  the fiber-graph $\widehat{\cal F}$ is a vertex $(\nu)$ such that  the image of $\widehat{\cal E}(\nu)$  by the injection $\cal I \colon V(\cal G) \to V(G)$  is a $\cal L$-node of the resolution graph $G$ of $\pi_{\ell}$ (see Section \ref{subsec:cal G-resolution}). 
 Equivalently,  $(\nu)$ is an $\cal L$-node of $\widehat{\cal F}$ if and only if   for all  $t \in (0,\eta)$, $(H \cap X) \cap \widehat{\cal F}_{\nu, t}\neq \emptyset$ where $H$ is a hyperplane section of $X$. 
 \end{definition} 
 
 The proof of Proposition \ref{cor:characterization of normal embedding} will use the following characterization of principal component in terms of fiber-graph and which follows immediately from the definitions. 
 
 \begin{lemma} \label{def:F'} Consider the subgraph $\widehat{\cal F}'$ of $\widehat{\cal F}$ defined as the union of all simple {paths} in $\widehat{\cal F}$ connecting  pairs of  vertices among $\cal L$- and $\cal P$-nodes. Let $\delta$ be a test arc. A component $\widehat{\delta}$ of $\ell^{-1}(\delta)$ is principal if and only if  the vertex $(\nu)$ of $\widehat{\cal F}$ such that $\delta(t) {\in} \widehat{F}_{\nu,t}$ is a vertex   of $\widehat{\cal F}'$. 
 \end{lemma}

 In the twin paper \cite{NPP2}, we prove that any minimal surface singularity is LNE by using Theorem \ref{cor:complex characterization of normal embedding}. In the following example, we illustrate the gain by using Proposition \ref{cor:characterization
  of normal embedding} instead of Proposition
\ref{prop:characterization of normal embedding2} on a specific minimal singularity. 

\begin{example}\label{ex:Minimal Singularity} Let $(X,0)$ be {a} minimal surface singularity with dual resolution
graph $\Gamma$ given in Example 1 of \cite{B1} and Example 5.6
of \cite{B2}. In the picture below we have to the left the dual
graph of the minimal resolution of $(X,0)$ that factors through both the
blow-up and the Nash modification, and on the right the resolution tree of the
discriminant. The arrows in the left graph indicates the strict
transform of the polar and on the right the strict transform of the
discriminant.
 
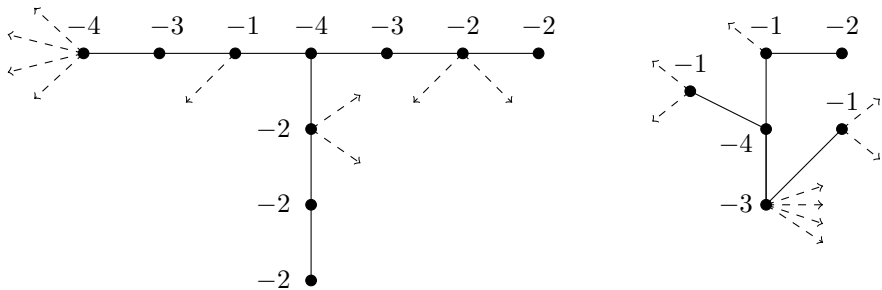
\begin{center}
 \begin{tikzpicture}
 
\draw[thin](-3,0)--(-2,0);
\draw[thin](-2,0)--(-1,0);
\draw[thin](-1,0)--(0,0);
\draw[thin](0,0)--(1,0);
\draw[thin](1,0)--(2,0);
\draw[thin](2,0)--(3,0);
\draw[thin](0,0)--(0,-1);
\draw[thin](0,-1)--(0,-2);
\draw[thin](0,-2)--(0,-3);

\draw[dashed,>-stealth,->](-1,0)--+(-.65,-.65);
\draw[dashed,>-stealth,->](0,-1)--+(.65,.45);
\draw[dashed,>-stealth,->](0,-1)--+(.65,-.45);
\draw[dashed,>-stealth,->](2,0)--+(.65,-.65);
\draw[dashed,>-stealth,->](2,0)--+(-.65,-.65);
\draw[dashed,>-stealth,->](-3,0)--+(-.65,.6);
\draw[dashed,>-stealth,->](-3,0)--+(-1,.25);
\draw[dashed,>-stealth,->](-3,0)--+(-1,-.25);
\draw[dashed,>-stealth,->](-3,0)--+(-.65,-.6);

\draw[fill=black](0,0)circle(2pt);
\node(a)at(0,0.35){{$-4$}};

\draw[fill=black](-2,0)circle(2pt);
\node(a)at(-1.9,0.35){{$-3$}};

\draw[fill=black](0,-2)circle(2pt);
\node(a)at(-.5,-2){{$-2$}};

\draw[fill=black](-3,0)circle(2pt);
\node(a)at(-3,0.35){{$-4$}};

\draw[fill=black](1,0)circle(2pt);
\node(a)at(1,0.35){{$-3$}};

\draw[fill=black](3,0)circle(2pt);
\node(a)at(3,0.35){{$-2$}};

\draw[fill=black](-1,0)circle(2pt);
\node(a)at(-.9,0.35){{$-1$}};

\draw[fill=black](2,0)circle(2pt);
\node(a)at(2,0.35){{$-2$}};

\draw[fill=black](0,-3)circle(2pt);
\node(a)at(-.5,-3){{$-2$}};

\draw[fill=black](0,-1)circle(2pt);
\node(a)at(-.5,-1){{$-2$}};

\draw[thin](6,-2)--(6,-1);
\draw[thin](6,-1)--(5,-.5);
\draw[thin](6,-2)--(7,-1);
\draw[thin](6,-2)--(6,-1);
\draw[thin](6,-1)--(6,0);
\draw[thin](6,0)--(7,0);

\draw[dashed,>-stealth,->](6,0)--+(-.5,.4);
\draw[dashed,>-stealth,->](5,-.5)--+(-.5,.4);
\draw[dashed,>-stealth,->](5,-.5)--+(-.5,-.4);
\draw[dashed,>-stealth,->](7,-1)--+(.5,-.4);
\draw[dashed,>-stealth,->](7,-1)--+(.5,.4);
\draw[dashed,>-stealth,->](6,-2)--+(.75,.25);
\draw[dashed,>-stealth,->](6,-2)--+(.75,0);
\draw[dashed,>-stealth,->](6,-2)--+(.75,-.25);
\draw[dashed,>-stealth,->](6,-2)--+(.75,-.5);

\draw[fill=black](6,-2)circle(2pt);
\node(a)at(5.6,-2){{$-3$}};

\draw[fill=black](6,-1)circle(2pt);
\node(a)at(5.6,-1.2){{$-4$}};

\draw[fill=black](5,-.5)circle(2pt);
\node(a)at(5,-.15){{$-1$}};

\draw[fill=black](6,0)circle(2pt);
\node(a)at(6,.35){{$-1$}};

\draw[fill=black](7,0)circle(2pt);
\node(a)at(7,.35){{$-2$}};

\draw[fill=black](7,-1)circle(2pt);
\node(a)at(7,-.65){{$-1$}};

\end{tikzpicture}
\end{center}

There are infinitely many
minimal singularities with $\Gamma$ as its dual resolution graph, but
as is shown in \cite{NPP2} they all have the same bilipschitz
geometry. 
 The discriminant of $(X,0)$ has equation 
$(x^4+y^4)(x^2+y^5)(x+y^2+iy^3)(x+y^2-iy^3)(y^2+x^4)=0$ (Example 5.6
of \cite{B2}). Notice that the graph on the right above is the graph of
the resolution $\rho'_\ell$. To get the graph of $\rho_\ell$ we have
to blow up all edges between $\Delta$-nodes. In this case the only
edge between $\Delta$-nodes is the edge between the root vertex and
adjacent vertex with weight $-1$. Hence in the graph above we have
blown up this edge to create a separation-node.
In the following picture, the left graph is 
$\widehat{\cal F}$  and the right graph is $\cal F$. The graph-map
$\widehat{\cal C}$ preserves the shape and colour of a vertex, i.e.,
$\widehat{\cal C}$ of a white vertex with thick boundary in $\widehat{\cal F}$
is the white vertex with thick boundary in $\cal F$ and so
on. All the vertices of $\cal F$ but the white one with thin boundary
are nodes. The subgraph of $\widehat{\cal F}$ consisting of the
vertices connected by the thick edges is $\widehat{\cal F}'$. The
graph $\widehat{\cal F}$ has $36$ vertices while $\widehat{\cal F}'$
has only  $12$ vertices. Consider a test arc  $\delta$ at the central
vertex of ${\cal F}$ with inner rate $2$ (its strict transform is
represented by an arrow) on the picture. Its lifting by $\ell$ has $6$
components, whose strict transform are represented by the 6 arrows in
$\widehat{\cal F}$. Therefore, to prove LNE by using Proposition 
\ref{prop:characterization of normal embedding2},  we would have to
test the arc criterion on each of the 15 pairs of arcs among these
components. Now,  only $3$ of the three arrows are attached to
$\widehat{\cal F}'$, i.e.,   correspond to principal components of
$\ell^{-1}(\delta)$. Therefore, to prove LNE by using  Proposition 
\ref{cor:characterization of normal embedding} instead of Proposition 
\ref{prop:characterization of normal embedding2}, we just have to test
the criterion on pairs of  arcs  among these $3$ principal components,
so we just have to test $3$ pairs of arcs instead of $15$. 

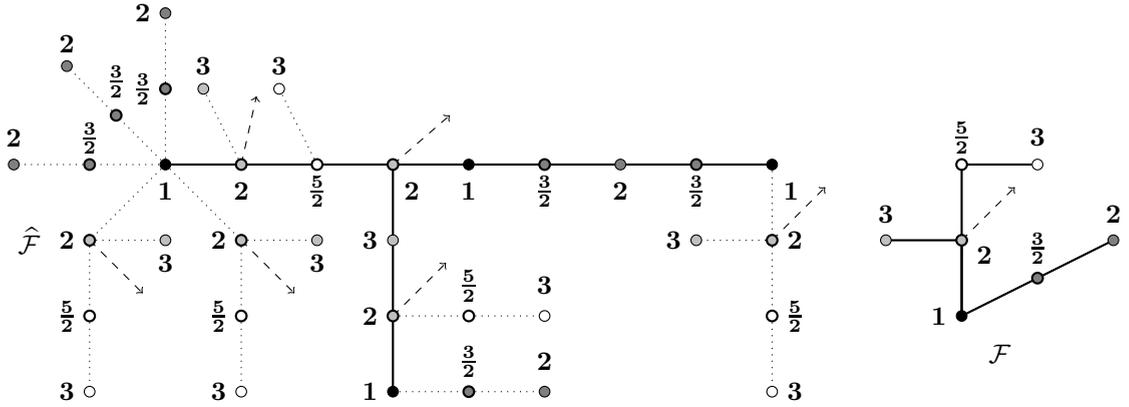
\begin{center}
 \begin{tikzpicture}
 
\draw[thick](-3,0)--(-2,0);
\draw[thick](-2,0)--(-1,0);
\draw[thick](-1,0)--(0,0);
\draw[thick](0,0)--(1,0);
\draw[thick](1,0)--(2,0);
\draw[thick](2,0)--(3,0);
\draw[thick](3,0)--(4,0);
\draw[thick](4,0)--(5,0);
\draw[thick](0,0)--(0,-1);
\draw[thick](0,-1)--(0,-2);
\draw[thick](0,-2)--(0,-3);

\draw[dotted](-2,0)--(-2.5,1);
\draw[dotted](-1,0)--(-1.5,1);
\draw[dotted](0,-2)--(1,-2);
\draw[dotted](1,-2)--(2,-2);
\draw[dotted](5,0)--(5,-1);
\draw[dotted](5,-1)--(5,-2);
\draw[dotted](5,-2)--(5,-3);
\draw[dotted](5,-1)--(4,-1);
\draw[dotted](-3,0)--(-2,-1);
\draw[dotted](-3,0)--(-4,-1);
\draw[dotted](-2,-1)--(-2,-2);
\draw[dotted](-2,-2)--(-2,-3);
\draw[dotted](-2,-1)--(-1,-1);
\draw[dotted](-4,-1)--(-4,-2);
\draw[dotted](-4,-2)--(-4,-3);
\draw[dotted](-4,-1)--(-3,-1);
\draw[dotted](-3,0)--(-3,1);
\draw[dotted](-3,0)--(-4,0);
\draw[dotted](-3,0)--(-3.65,.65);
\draw[dotted](-4,0)--(-5,0);
\draw[dotted](-3,1)--(-3,2);
\draw[dotted](-3.65,.65)--(-4.3,1.3);
\draw[dotted](0,-3)--(1,-3);
\draw[dotted](1,-3)--(2,-3);

\draw[dashed,>-stealth,->](0,0)--+(.75,.65);
\draw[dashed,>-stealth,->](0,-2)--+(.7,.7);
\draw[dashed,>-stealth,->](5,-1)--+(.7,.7);

\draw[dashed,>-stealth,->](-2,0)--+(.2,.9);
\draw[dashed,>-stealth,->](-2,-1)--+(.7,-.7);
\draw[dashed,>-stealth,->](-4,-1)--+(.7,-.7);

\draw[fill=lightgray,thick](0,0)circle(2pt);
 \node(a)at(0.25,-0.35){{$\bf 2$}};

\draw[fill=lightgray,thick](-2,0)circle(2pt);
\node(a)at(-2,-0.35){{$\bf 2$}};

\draw[fill=lightgray,thick](0,-2)circle(2pt);
\node(a)at(-.3,-2){{$\bf 2$}};

\draw[fill=black](-3,0)circle(2pt);
\node(a)at(-3,-0.35){{$\bf 1$}};

\draw[fill=black](1,0)circle(2pt);
\node(a)at(1,-0.35){{$\bf 1$}};

\draw[fill=black](5,0)circle(2pt);
\node(a)at(5.25,-0.35){{$\bf 1$}};

\draw[fill=white,thick](-1,0)circle(2pt);
\node(a)at(-1,-0.35){{$\bf\frac{5}{2}$}};

\draw[fill=gray](3,0)circle(2pt);
\node(a)at(3,-0.35){{$\bf 2$}};

\draw[fill=gray,thick](2,0)circle(2pt);
\node(a)at(2,-0.35){{$\bf\frac{3}{2}$}};

\draw[fill=gray,thick](4,0)circle(2pt);
\node(a)at(4,-0.35){{$\bf\frac{3}{2}$}};

\draw[fill=black](0,-3)circle(2pt);
\node(a)at(-.3,-3){{$\bf 1$}};

\draw[fill=lightgray](0,-1)circle(2pt);
\node(a)at(-.3,-1){{$\bf 3$}};

\draw[fill=lightgray](-2.5,1)circle(2pt);
 \node(a)at(-2.5,1.3){{$\bf 3$}};

\draw[fill=white](-1.5,1)circle(2pt);
 \node(a)at(-1.5,1.3){{$\bf 3$}};
 
\draw[fill=white,thick](1,-2)circle(2pt);
 \node(a)at(1,-1.6){{$\bf\frac{5}{2}$}};
 
\draw[fill=gray,thick](1,-3)circle(2pt);
 \node(a)at(1,-2.6){{$\bf\frac{3}{2}$}};

\draw[fill=gray](2,-3)circle(2pt);
 \node(a)at(2,-2.6){{$\bf 2$}};
 
\draw[fill=lightgray,thick](5,-1)circle(2pt);
 \node(a)at(5.3,-1){{$\bf 2$}};
 
\draw[fill=white,thick](5,-2)circle(2pt);
 \node(a)at(5.3,-2){{$\bf\frac{5}{2}$}};

\draw[fill=white](5,-3)circle(2pt);
 \node(a)at(5.3,-3){{$\bf 3$}};

\draw[fill=lightgray](4,-1)circle(2pt);
 \node(a)at(3.7,-1){{$\bf 3$}};
 
\draw[fill=lightgray,thick](-2,-1)circle(2pt);
 \node(a)at(-2.3,-1){{$\bf 2$}};

\draw[fill=white,thick](-2,-2)circle(2pt);
 \node(a)at(-2.3,-2){{$\bf\frac{5}{2}$}};

\draw[fill=white](-2,-3)circle(2pt);
 \node(a)at(-2.3,-3){{$\bf 3$}};
 
\draw[fill=lightgray](-1,-1)circle(2pt);
 \node(a)at(-1,-1.3){{$\bf 3$}};
 
\draw[fill=lightgray,thick](-4,-1)circle(2pt);
 \node(a)at(-4.3,-1){{$\bf 2$}};

\draw[fill=white,thick](-4,-2)circle(2pt);
 \node(a)at(-4.3,-2){{$\bf\frac{5}{2}$}};
 
\draw[fill=white](-4,-3)circle(2pt);
 \node(a)at(-4.3,-3){{$\bf 3$}};

\draw[fill=lightgray](-3,-1)circle(2pt);
 \node(a)at(-3,-1.3){{$\bf 3$}};

\draw[fill=gray](-5,-0)circle(2pt);
 \node(a)at(-5,.35){{$\bf 2$}};

\draw[fill=gray](-3,2)circle(2pt);
 \node(a)at(-3.3,2){{$\bf 2$}};
 
\draw[fill=gray](-4.3,1.3)circle(2pt);
 \node(a)at(-4.3,1.6){{$\bf 2$}};

\draw[fill=gray,thick](-3.65,.65)circle(2pt);
 \node(a)at(-3.65,1.1){{$\bf\frac{3}{2}$}};

\draw[fill=gray,thick](-3,1)circle(2pt);
 \node(a)at(-3.3,1){{$\bf\frac{3}{2}$}};

\draw[fill=gray,thick](-4,0)circle(2pt);
 \node(a)at(-4,0.35){{$\bf\frac{3}{2}$}};
 
\draw[fill=white](2,-2)circle(2pt);
 \node(a)at(2,-1.6){{$\bf 3$}};

\node(a)at(-4.8,-1){{$\widehat{\cal F}$}};

\draw[thick](7.5,-2)--(7.5,-1);
\draw[thick](7.5,-1)--(6.5,-1);
\draw[thick](7.5,-2)--(8.5,-1.5);
\draw[thick](7.5,-2)--(7.5,-1);
\draw[thick](7.5,-1)--(7.5,0);
\draw[thick](8.5,-1.5)--(9.5,-1);
\draw[thick](7.5,0)--(8.5,0);

\draw[dashed,>-stealth,->](7.5,-1)--+(.7,.7);

\draw[fill=black](7.5,-2)circle(2pt);
\node(a)at(7.2,-2){{$\bf 1$}};

\draw[fill=lightgray,thick](7.5,-1)circle(2pt);
\node(a)at(7.8,-1.2){{$\bf 2$}};

\draw[fill=lightgray](6.5,-1)circle(2pt);
\node(a)at(6.5,-.65){{$\bf 3$}};

\draw[fill=white,thick](7.5,0)circle(2pt);
\node(a)at(7.5,.36){{$\bf\frac{5}{2}$}};

\draw[fill=white](8.5,0)circle(2pt);
\node(a)at(8.5,.36){{$\bf 3$}};

\draw[fill=gray,thick](8.5,-1.5)circle(2pt);
\node(a)at(8.5,-1.1){{{$\bf\frac{3}{2}$}}};

\draw[fill=gray](9.5,-1)circle(2pt);
\node(a)at(9.5,-.65){{{$\bf 2$}}};

\node(a)at(8,-2.5){{$\cal F$}};

\end{tikzpicture}
\end{center}
\end{example}

Notice that each component of the  complementary subgraph $  \widehat{\cal F} \setminus \widehat{\cal F}'$ is a  rooted tree oriented from its root by  strictly increasing inner rates and attached to $\widehat{\cal F}'$ by a  single edge adjacent to its root.  

The proof of Proposition \ref{cor:characterization of normal embedding} needs the following  key Lemma.

\begin{lemma}  \label{lem:partner3} Let $(j')$ be a vertex of $T$.
Let $T'$ be a maximal connected subgraph of $T$ whose vertices
$(k)\neq (j')$ satisfy that the simple path from $(k)$ to the
  root vertex passes through $(j')$. Let   $\widehat{\cal F}_1$ and $\widehat{\cal F}_2$ be two distinct components of $ \widehat{\cal F} \setminus (\cal E \circ \widehat{\cal C})^{-1}(j')$  with  $(\cal E \circ \widehat{\cal C})(\widehat{\cal F}_1)=(\cal E \circ \widehat{\cal C})(\widehat{\cal F}_2)=T'$. 

 Assume that for all pairs of vertices $(\nu'_1)$ and $(\nu'_2 )$   adjacent to $\widehat{\cal F}_1$ and  $\widehat{\cal F}_2$ respectively, all generic projections $\ell  \colon (X,0) \to (\C^2,0)$ and all test arcs  $\delta'$  at $(j')$, any pair of components $\delta'_1, \delta'_2$ over $\delta'$ with $\delta'_i(t) \in \widehat{F}_{t,\nu'_i}$  satisfies the arc criterion. Then for all generic projection $\ell  \colon (X,0) \to (\C^2,0)$ and for all test arcs  $\delta$  at a vertex $(j)$ of $T'$, any pair of arcs  $\delta_1, \delta_2$ over $\delta$ with $\delta_i(t) \in \widehat{\cal F}_{t, \nu_i}$ with $(\nu_i) \in V(\widehat{\cal F}_i) $   satisfies the arc criterion. 
\end{lemma}

 \begin{proof} Let $\widehat{\cal F}_1$ and $\widehat{\cal F}_2$ be the two components of $ \widehat{\cal F} \setminus (\cal E \circ \widehat{\cal C})^{-1}(j')$ which contain  $(\nu_1)$ and $(\nu_2)$ respectively. First, notice that if $(\nu'_1)$ and $(\nu''_1)$ are two vertices of $(\cal E \circ \widehat{\cal C})^{-1}(j')$ adjacent  to $\widehat{\cal F}_1$, then there is a simple path inside $\widehat{\cal F}_1$ from 
  $(\nu'_1)$ and $(\nu''_1)$ and this path then has its inner rates all $\geq q_{j'}$. The same is true for two vertices adjacent to  $\widehat{\cal F}_2$ with a path inside $\widehat{\cal F}_2$. Therefore, if 
 $(\nu'_1)$  and $(\nu'_2)$ are two vertices of $(\cal E \circ \widehat{\cal C})^{-1}(j')$ adjacent respectively to $\widehat{\cal F}_1$ and $\widehat{\cal F}_2$, then $q_{\nu'_1, \nu'_2}$ does not depend on the choice of $(\nu'_1)$  and $(\nu'_2)$. We set $q=q_{\nu'_1, \nu'_2}$. 
 
 Notice that when $q < q_{j'}$, the proof of the lemma is a direct application of \ref{partner2} of Lemma \ref{lem:partner2}.  We now  give  the proof in all cases.   
 
We have $q_{\nu_1}= q_{\nu_2} =q_j >q_{j'}$ and for $i=1,2$, any simple path from $(\nu_i)$ to a vertex of $(\cal E \circ \widehat{\cal C})^{-1}(j')$ has  strictly  decreasing
rates. Moreover, any simple path $p$ from $(\nu_1)$ to $(\nu_2)$ in $\widehat{\cal F}$ must go through a vertex $(\nu_1)$ adjacent to  $\widehat{\cal F}_1$ and then to a vertex $(\nu_2)$ adjacent to $\widehat{\cal F}_2$ (maybe $(\nu_1)=(\nu_2)$). This implies that  $q_{\nu_1,\nu_2}= q$ in the graph $\widehat{\cal F}$. 
 
 Let $\delta'$ be a test curve at $(j')$, let $\delta'_1, \delta'_2$   be  the partner pair of  $\delta_1, \delta_2$ over $\delta'$ and let $(\nu'_1)$  and $(\nu'_2)$ be the two vertices of $(\cal E \circ \widehat{\cal C})^{-1}(j')$ such that  $\delta'_1(t)  \in \widehat{F}_{t,\nu'_1}$ and $\delta'_2(t)  \in \widehat{F}_{t,\nu'_2}$.  By Lemma \ref{lem:inner X} we therefore obtain  $q_i(\delta'_1,\delta'_2) = q_{\nu'_1,\nu'_2} = q_{\nu_1, \nu_2}=q_i(\delta_1,\delta_2) $, so 
  $q_i(\delta_1,\delta_2)=q$. 
  
 Let us now prove that $q_o(\delta'_1,\delta'_2) = q$. 
 
 Choose coordinates of $\C^2$ so that $\ell=(x,y)$ and $F_t = \{x=t\}$ with $t \in \R^+$.  For $i=1,2$, let $\widehat{S}_{t,i}$ be the  surface inside $\widehat{F}_t$ defined  as $\bigcup_{v \in V(\widehat{\cal F}_i)} \widehat{F}_{t,v}$
 union the  intermediate annuli corresponding to the edges of $\widehat{\cal F}_i$ and the edges joining $\widehat{\cal F}_i$ to vertices of $(\cal E \circ \widehat{\cal C})^{-1}(j')$. The two surfaces  $\widehat{S}_{t,1}$ and $\widehat{S}_{t,2}$ are disjoint. By  hypothesis,   any  test curve $\delta'$ at $(\cal E \circ \widehat{\cal C})(\nu_1')$ satisfies  the arc criterion, i.e., for any pair of components $\delta_1, \delta_2$ of $\ell^{-1}(\delta')$, we have $q_o(\delta'_1, \delta'_2) = q_i(\delta'_1, \delta'_2)= q$. In particular, this holds for any test curve $\delta'$ such that $\delta'(t) \in \partial S_t$, where $S_t$ is the disc in $F_t$ defined by $S_t = \ell(\widehat{S}_{t,1}) =  \ell(\widehat{S}_{t,2})$ and  $\partial $ means boundary. Then, using the same arguments as in the proof of Proposition \ref{liftingoftestarcs}, we obtain that  the outer distance between the boundaries $ \partial \widehat{S}_{t,1}$ and   $ \partial \widehat{S}_{t,2}$ is a $\Theta(t^{q})$. Therefore, the images $\ell'(\widehat{S}_{t,1})$ and $\ell'(\widehat{S}_{t,2})$ by a generic projection $\ell' \colon (X,0) \to (\C^2,0)$ given by $\ell'=(x,z)$ consist of two discs of diameter $ \Theta(t^{q})$ at  distance $\Theta(t^{q})$  from each other  inside $\ell'(F_{\nu', t}) \subset \{x=t\}$. Since $\ell'(\delta_1) \subset  \ell'(\widehat{S}_{t,1})$ and  $\ell'(\delta_2) \subset  \ell'(\widehat{S}_{t,2})$, we  then  have $q_o({\ell'}(\delta_1), {\ell'}(\delta_2))=q$. 
   
 Let $\gamma$  be a plane curve germ such that $\delta$ is a real slice of $\gamma$ and take for  ${\ell'} \colon (X,0) \to (\C^2,0)$  a projection  which is also generic for the complex curve $\gamma_0 = \ell^{-1}(\gamma)$. Then the restriction ${\ell'} \mid_{\gamma_0} \colon \gamma_0 \to {\ell'}(\gamma_0)$ is a bilipschitz homeomorphism for the outer metric (Theorem \ref{generic projection bilipschitz}). In particular,  we have  $q_o(\delta_1, \delta_2) = q_o({\ell'}(\delta_1), {\ell'}(\delta_2))$. Therefore $q_o(\delta_1, \delta_2) =q$. 

Summarizing, we then obtain   $q_o(\delta'_1, \delta'_2) = q = q_i(\delta'_1, \delta'_2)$, proving the lemma. 
 \end{proof}

\begin{proof}[of Proposition \ref{cor:characterization of normal embedding}]
The ``only'' if direction is a direct consequence of   Proposition \ref{prop:characterization of normal embedding2}. Hence we only  need to prove the ``if''  direction. Assume that for all  generic projections $\ell$,  any  pair of principal components over a nodal  test arc  for $\ell$ satisfies the arc criterion. 
By Proposition \ref{prop:characterization of normal embedding2} we just need to show that this implies that for all test arcs $\delta$ for $\ell$, 
any pair of components $\delta_1,\delta_2$ of $\ell^{-1}(\delta)$
also satisfies  the arc criterion $q_i(\delta_1,\delta_2) = q_o(\delta_1,\delta_2)$. 
 
By assumption, any pair of principal components over a nodal test arc
satisfies the arc criterion. Hence we first assume that $\delta$  is
a   test curve at a vertex $(j)$ of $T$ which is not a node, and that
 $\delta_1$ and $\delta_2$ are principal components. For $i=1,2$, let $(\nu_i)$ be  the vertex of $\widehat{\cal F}$ such that $\delta_i\cap
\widehat{F}_{t,\nu_i} \neq \emptyset$.  Let   $\cal S$ be the string in $T$ minus the nodes   which contains $(j)$. Since $\ell^{-1}(\delta)$ contains a principal component, then the two vertices $(j')$ and $(j'')$ adjacent to $\cal S$ are nodes. We assume $q_{j'} < q_{j} < q_{j''}$.  Let $\cal S_1$ and $\cal S_2$ be the two strings in $\widehat{\cal F}$ containing respectively $(\nu_1)$ and $(\nu_2)$ and such that   $ (\cal E \circ \widehat{\cal C})(\cal S_i)=\cal S$. For each $i=1,2$, let  $(\nu'_i)$ and $(\nu''_i)$ be the two nodes of $\widehat{\cal F}$  adjacent to $\cal S_i$ which map respectively to $(j')$ and $(j'')$ by $\cal C \circ \widehat{\cal E}$ (we may have $\nu'_1 = \nu'_2$ or $\nu''_1 = \nu''_2$). 

Assume first that $q_{\nu_1, \nu_2}=  q_{j}$. Since the
   strings $\cal S_1$ and $\cal S_2$ are disjoint except perhaps at
   their extremities, this means that there is a simple  path from
   $(\nu''_1)$ to $(\nu''_2)$ with rates $\geq q_{j''}$ and that
   $q_{\nu'_1, \nu'_2}=  q_{j'}$ and   $q_{\nu''_1, \nu''_2}=
   q_{j''}$.  By definition of $\rho_{\ell}$, at least one of $(j')$
 and $(j'')$ is not a $\Delta$-node. Assume  that $(j')$  is not a
 $\Delta$-node. Let $\delta'$ be a test curve at $(j')$ and let
 $(\delta'_1, \delta'_2)$ be the partner pair of $(\delta_1,
 \delta_2)$  over $\delta'$. Since $\delta_1$ and $\delta_2$
   are principal it follows that $\delta_1'$ and $\delta_2'$ are
   principal and hence by hypothesis $(\delta'_1, \delta'_2)$ satisfies the arc criterion. Then, by Lemma  \ref{LNE along strings}, the values of the Gauss map $\widetilde{\lambda}$ at $(\nu'_1)$ and $(\nu'_2)$ are distinct. Since for each $i=1,2$, the vertices $(\nu'_i)$ and $(\nu_i)$ are in the same connected component of  $\widehat{\cal F}$ minus its $\cal P$-nodes, then they correspond to the same value of the Gauss-map. Summarizing, we obtain that  $\widetilde{\lambda}$ has distinct values at $(\nu_1)$ and $(\nu_2)$. Applying Lemma  \ref{LNE along strings} in the other direction, we  obtain that $(\delta_1, \delta_2)$ satisfies the arc criterion. When  $(j')$ is a $\Delta$-node, then $(j'')$ is not a $\Delta$-node and one proves   that $(\delta_1, \delta_2)$ satisfies the arc criterion  by the same arguments as in the previous case. 
 
 Assume now  that $q_{\nu_1, \nu_2}< q_{j}$, so we have
   $q_{\nu_1, \nu_2} \leq q_{j'}$. Since $(j)$ and $(j'')$ are in the
 same connected component of $T \setminus (j')$ and since the pair
 $(\delta'_1, \delta'_2)$ satisfies the arc criterion, then
 $(\delta_1, \delta_2)$ also satisfies the arc criterion by Lemma
 \ref{lem:partner2} \ref{partner2}. 

Since we have proved that all pairs of principal components satisfy
 the arc criterion, we must prove that if one of $\delta_1$ and $\delta_2$, say $\delta_1$, is not principal, then the pair $\delta_1,\delta_2$ also satisfies $q_i(\delta_1,\delta_2) = q_o(\delta_1,\delta_2)$. 
Let $(\nu_i)$ be the vertex of $\widehat{\cal F}$ such that $\delta_i\cap
\widehat{F}_{t,\nu_i} \neq \emptyset$. Since $\delta_1$ is not principal,  $(\nu_1)$ is in the  complementary subgraph $ \widehat{\cal F} \setminus \widehat{\cal F}' $. Let  $\widehat{\cal F}_1$  be connected component of $ \widehat{\cal F} \setminus \widehat{\cal F}' $  containing $(\nu_1)$. It is a  tree  and there is a unique vertex  $(\nu_1')$ of $\widehat{\cal F}'$ adjacent to $\widehat{\cal F}_1$.  Then  $(\nu_1')$ is a node of $\widehat{\cal F}$ and   the vertex $(j') = (\cal C \circ \widehat{\cal E})(\nu'_1)$ is a node of $T$. We use again the notations introduced in the proof of Lemma  \ref{lem:partner3}: let $\widehat{S}_{t,1}$ be the part of $\widehat{F}_t$ associated with $\widehat{\cal F}_1$ and set $S_t = \ell(\widehat{S}_{t,1})$. Let $\Pi$ be the polar curve of $\ell$.   Since $\Pi\cap \widehat{S}_{t,1}= \emptyset$, the restriction  $\ell\vert_{\widehat{S}_{t,1}} \colon \widehat{S}_{t,1} \to   S_t $ is a regular connected covering over a disc, and hence, an isomorphism. This implies that  $(\nu_2)$ is not a vertex of $ \widehat{\cal F}_1$, so it is in a connected component  $\widehat{\cal F}_2$ of $\widehat{\cal F} \setminus (\cal E \circ \widehat{\cal C})^{-1}(j')$ distinct from $\widehat{\cal F}_1$. 

Assume  first that there is a vertex $(\nu'_2)$  adjacent to $\widehat{\cal F}_2$ which is in $\widehat{\cal F}'$. Then all vertices adjacent to $\widehat{\cal F}_2$ also are in $\widehat{\cal F}'$. Since $(\nu'_1)$ is in $\widehat{\cal F}'$, then all vertices adjacent to  $\widehat{\cal F}_1$ also are in  $\widehat{\cal F}'$.  Therefore, by hypothesis, any pair of components $\delta'_1, \delta'_2$ at such vertices over any test curve at $(j')$ satisfies the arc criterion. Then, applying  Lemma \ref{lem:partner3}  we obtain that $(\delta_1, \delta_2)$ also satisfies the arc criterion. 

Assume now that all vertices  $(\nu'_2)$  adjacent to $\widehat{\cal F}_2$ are in $\widehat{\cal F} \setminus \widehat{\cal F}'$, i.e., $\delta'_2$ is not principal. The component  $\widetilde{\cal F}_2$ of  $\widehat{\cal F} \setminus \widehat{\cal F}'$ containing $(\nu_2)$ is a tree and there is a unique vertex  $(\nu''_2)$ of $\widehat{\cal F}'$ adjacent to $\widetilde{\cal F}_2$.  Set $(j'') = (\cal E \circ \widehat{\cal C})(\nu''_2)$. Then $(\nu''_2)$ is a node of $\widehat{\cal F}$. Let  $\widetilde{\cal F}_1$ be the connected  component of $\widehat{\cal F} \setminus (\cal E \circ \widehat{\cal C})^{-1}(j'')$ containing $(\nu_1)$.  
 Since  $(\nu'_1)$ is in $ \widehat{\cal F}'$ and since  $(\nu'_1) \in \widetilde{\cal F}_1$,  then  any vertex  $(\nu''_1)$ adjacent to   $ \widetilde{\cal F}_1$  is also  in $\widehat{\cal F'}$. 
 
 Therefore, by hypothesis, any pair of components $\delta'_1, \delta'_2$ at such vertices $(\nu''_1)$ and $(\nu''_2)$ over any test curve at $(j'')$ satisfies the arc criterion. Then, applying twice Lemma \ref{lem:partner3} we obtain first  that any pair $(\delta'_1, \delta'_2)$  at vertices $(\nu'_1)$  and $(\nu'_2)$ over any test curve $\delta'$ at $(j')$ satisfies the arc criterion, and then that $(\delta_1, \delta_2)$ also satisfies the arc criterion. 

We then have proved that for all test {curves} $\delta$, any pair of
components $\delta_1,\delta_2$ of  $\ell^{-1}(\delta)$ where
at least one of $\delta_1,\delta_2$ is not principal satisfies the
arc criterion.
\end{proof}

\section{Inner contacts between complex curves on a normal surface}\label{sec: inner
  contact on a normal surface 2}

We need the following proposition in the proof of Theorem  \ref{cor:complex characterization of normal embedding}.
\begin{proposition} \label{lem:complex real rates}  Let $(\gamma,0)$ and $(\gamma',0)$ be two irreducible complex curves on $(X,0)$  tangent to the same complex line parametrized by $x \in \C$, let $\delta_1,\ldots,\delta_r$ be the components of the real slice $\gamma \cap \{x=t \in \R^+\}$ and let $\delta'_1,\ldots,\delta'_s$ be those of $\gamma' \cap \{x=t \in \R^+\}$. Then we have: 
$$q_{out}(\gamma, \gamma') = \max_{k,l}  q_o(\delta_k,  {\delta'_l}) \hskip0,5cm  \hbox{ and } \hskip0,5cm  q_{inn}(\gamma, \gamma') = \max_{k,l}  q_i(\delta_k,  {\delta'_l}).$$
\end{proposition}

\begin{proof}  Assume $(X,0) \subset (\C^n,0)$ and call $(x,x_2,\ldots,x_n)$ the coordinates of $\C^n$. We can assume that the tangent lines at $0$ to $\gamma$ and $\gamma'$ are not included in the  hyperplane $\{x=0\}$.  Therefore the outer contact  $q_{out}(\gamma, \gamma')$ can be computing by taking  intersection with balls with corners $\Bbb S^1_t \times B^{2n-2}_\alpha$ with $\alpha>0$ sufficiently large instead of standard spheres $\Bbb S_t^{2n-1}$, i.e.,  we have $q_{out}(\gamma, \gamma') = ord_t \ d_o(\gamma \cap \{|x|=t\}, \gamma' \cap \{|x|=t\})$.

Moreover,  $\gamma$ and $\gamma'$ admit Puiseux expansions  $(x,x_2,\ldots,x_n) = (x, f_2(x), \ldots,f_n(x))$ with $f_i \in \C\{x^{1/r}\}$ and $(x,x_2\ldots,x_n) = (x, g_2(x), \ldots,g_n(x))$ with $g_i \in \C\{x^{1/s}\}$ respectively. Then the intersection $\gamma \cap \{|x|=t\}$ is a uniform braid obtained as the trajectory of the $r$ points $\gamma \cap \{x=t\}$ over the circle $\Bbb S^1_t = \{x=t e^{i \theta}\colon \theta \in \R\}$ through the Puiseux expansion, and we have the same for $\gamma'$. Therefore, $q_{out}(\gamma, \gamma')$ can be computed by measuring distances inside the hyperplane $\{x=t\}$, i.e., $q_{out}(\gamma, \gamma') = ord(t) \ d_o(\gamma \cap \{x_1=t\}, \gamma' \cap \{x=t\})$. Since by definition $\gamma \cap \{x=t\} = \{ \delta_k(t), k=1 \ldots,r \} $ and   $\gamma'\cap \{x=t\}  = \{ \delta'_l(t), l=1 \ldots,s \} $, we  then obtain $q_{out}(\gamma, \gamma') = \max_{k,l}  q_{out}(\delta_k, {\delta'_l})$. 

When working in $\C^2$, inner and outer distances coincide and we then have $q_{inn}(\gamma, \gamma') = \max_{k,l}  q_i(\delta_k, {\delta'_l})$. When $n > 2$,  then we can extend the above argument working with a  geometric decomposition of $(X,0)$. Consider  a generic  projection $\ell \colon (X,0) \to (\C^2)$ which is also generic for the curve $\gamma \cup \gamma'$ and let $\rho \colon Y \to \C^2$ be the minimal sequence of blow-ups of points which resolves the basepoints of the family of projected polar curves $(\ell(\Pi_{\cal D}))_{\cal D \in \Omega}$ and which resolves the curve $\ell(\gamma) \cup \ell(\gamma')$. Then consider the geometric decomposition of $(X,0)$ obtained by lifting by $\ell$ the pieces of the geometric decomposition associated with $\rho$. Then  $\gamma$ and $\gamma'$ are trajectories of the $r$ points  $\gamma \cap \{x=t\}$ and $\gamma' \cap \{x=t\}$ inside the $B$-pieces containing them, and since the pieces of the geometric decomposition are fibered by their  intersections with  $\{x=t\}$, which shrink faster than linearly when $t$ tends to $0$, then $q_{out}(\gamma, \gamma')$ can be computed by measuring inner distance between $ \gamma$ and $\gamma'$ inside the fibers  $ \widehat{F}_t= X \cap \{x=t\}$. This proves $q_{inn}(\gamma, \gamma') = \max_{k,l}  q_i(\delta_k, {\delta'_l})$. 
\end{proof}

\section{The LNE$_{test}$-resolution}

In this section, we construct a resolution $\mu_0 \colon W_0 \to X$ which will be used in the proof of  the ``only if" direction of Theorem 
  \ref{cor:complex characterization of normal embedding}. It has the following property:  it is a good resolution for every principal component over every nodal test curve for every  generic projection.
  
We use again the notations introduced in Section \ref{subsec:cal G-resolution}.  Let $\xi_{\ell} \colon X'_{\ell}  \to W_{\ell}$  be the morphism obtained by blowing down iteratively the exceptional $(-1)$-curves which are not on simple paths joining vertices of $G_0'$  (Definition \ref{def:G'}). We then obtain a resolution  $\mu_{\ell} \colon W_{\ell} \to X$ which factors through $\pi_0 \colon X_0 \to X$ by a morphism $\beta_{\ell} \colon W_{\ell} \to X_0$. 

\begin{lemma}  \label{lem:W_0}
The morphism $\beta_{\ell}$ does not depend on $\ell$.  We set $\beta_0 = \beta_{\ell}$, $W_0 = W_{\ell}$ and $\mu_0 = \mu_{\ell}$, so we have a commutative diagram: 
  $$\xymatrix{ 
   X'_{\ell}      \ar[dd]_{\widetilde{\ell}}   \ar[r]^{\xi_{\ell}}  &W_0  \ar[dr]_{\mu_0}   \ar[r]^{ \beta_0}   &X_0 \ar[d]_{\pi_0}   \\
  & & X  \ar[d]_{{\ell}}     \\
    Z_{\ell}\ar[rr]^{\rho_{\ell}} && \C^2   
 }$$
\end{lemma}

\begin{proof}   $\beta_{\ell}$ is a sequence of blow-ups of points which are all double points of the successive exceptional divisors. Since the families of curves $(\Pi_{\cal D})_{\cal D \in \Omega}$ and $(\ell_{\cal D'}(\Pi_{\cal D}))_{\cal D, \cal D'  \in \Omega \times \Omega}$ are equisingular in terms of strong simultaneous resolution, this sequence of double points does not depend on the choice of the generic $\ell$.  \end{proof}

\begin{definition} We call $\mu_0 \colon W_0 \to X$ the LNE$_{test}$-resolution of $(X,0)$. We denote by $\Gamma_0$ the graph of $\mu_0$ and by $\Gamma'_0$ the subgraph of $\Gamma_0$ which consists of the union of simple path joining $\cal L$- or $\cal P$-nodes. 
\end{definition}

By construction, we have the following characterization of  the principal components over nodal test curves:

\begin{lemma} \label{lem:W_0-2}  Let $\ell$ be a generic projection and let $\gamma$ be a nodal test curve for $\ell$. A component $\widehat{\gamma}$ of $\ell^{-1}(\gamma)$ is principal if and only if its strict transform by $\mu_0$ is a curvette of a component $E_v$ of $\mu_0^{-1}(0)$ such that $(v) \in V(\Gamma'_0)$.
\end{lemma}

\section{Proof of Theorem   \ref{cor:complex characterization of normal embedding}} \label{sec:proof main}

\begin{proof}[of the ``only if" direction of Theorem 
  \ref{cor:complex characterization of normal embedding}] 

    Assume $(X,0)$ is
  LNE. Let $\ell \colon (X,0) \to (\C^2,0)$ be
  a generic projection.  Let $\gamma_1$ and $\gamma_2$ be a pair of
  complex curves on $(X,0)$ such that $\ell(\gamma_1)=\ell(\gamma_2)$
  and $\ell(\gamma_i)^* \cap \Pi^* = \emptyset$. If $\gamma_1$ and $\gamma_2$ do not have a common tangent line, then $q_{out}(\gamma_1, \gamma_2) = 1$ and then $q_{inn}(\gamma_1, \gamma_2) = q_{out}(\gamma_1, \gamma_2) = 1$. If $\gamma_1$ and $\gamma_2$ have a common tangent line, let $\delta^{(1)}$ (resp.\ $\delta^{(2)}$) be a component of the real slice of $\gamma_1$ (resp.\ $\gamma_2$) as in Lemma \ref{lem:complex real rates}.   Since $(X,0)$ is LNE, then
  $q_i(\delta^{(1)}, \delta^{(2)}) = q_o(\delta^{(1)}, \delta^{(2)})$ by Theorem
  \ref{thm:arc criterion}, and then by  Lemma  \ref{lem:complex real rates} we get 
  $q_{inn}(\gamma_1,\gamma_2) = q_{out}(\gamma_1,\gamma_2)$ and
  Condition ($2^*$) is satisfied.

 We now consider a nodal test curve $(\gamma, 0)$ and a principal component $\widehat{\gamma}$  of $\ell^{-1}(\gamma)$. Let us  prove that Condition ($1^*$) is satisfied, i.e., that $mult(\widehat{\gamma}) =  mult(\gamma)$.
  
  Let $(i)$ be the node of $T$ such that $\gamma^*$ is a curvette of $C_{i}$.  The strict transform $\widehat{\gamma}^*$ of $\widehat{\gamma}$ by
 $ \mu_0$ is a curvette of an irreducible component $E$ of
 $ ( \mu_0)^{-1}(0)$ (Lemma \ref{lem:W_0-2}). Let us again denote by $E$ the irreducible curve in $X'_{\ell}$ which maps surjectively on $E$ by  $\xi_{\ell}$. Then we have $\widetilde{\ell} (E)=C_{i}$.
  
 Let $\ell' \colon (X,0) \to (\mathbb C^2,0)$ be another generic
 projection for $(X,0)$ which is also generic for the curve
 $\ell^{-1}(\gamma)$ (Definition \ref{def:generic projection
   curve}).  By  Lemma \ref{lem:W_0}, $\widetilde{\ell'}(E)$ is  the component  $C'_{i}$ of $\rho_{\ell'}^{-1}(0)$ corresponding to the node $(i)$ of $T$.  We then have 
   $q_{C_i} = q_{C'_i}$.
      
   The strict transform of $\ell'(\widehat{\gamma})$ by $\rho_{\ell'}$
   intersects $C'_{i}$ at a smooth point $p$ of $(\rho_{\ell'})^{-1}(0)$.  Let
   $\gamma'$ be the $\rho_{\ell'}$-image of a curvette of $C'_{i}$ which meets
   $C'_{i}$ at a point distinct from $p$. We then have
   $q_{out}(\ell'(\widehat{\gamma}), \gamma') = q_{C'_i} = q_{C_i}$, so there are
   Puiseux expansions of $\ell'(\widehat{\gamma})$ and $\gamma'$ which
   coincide for exponents $<q_{C_i}$ and which have distinct coefficients
   for $x^{q_{C_i}}$. 

Assume that  $mult(\widehat{\gamma}) \neq  mult(\gamma)$, i.e.,  $\widehat{\gamma}$ does not  satisfy Condition   ($1^*$).
  Then $mult(\widehat{\gamma}) = k\ mult(\gamma)$ where $k$ is the degree of the restriction $\ell \mid_{\widehat{\gamma}} \colon (\widehat{\gamma},0) \to
   (\gamma,0)$. Since $\ell'$ is a generic projection for
$\widehat{\gamma}$, then it is  a bilipschitz
 homeomorphism from  $\widehat{\gamma}$ to
   $\ell'( \widehat{\gamma})$ for the outer metric (Theorem
   \ref{generic projection bilipschitz}) so these two curves have
 same multiplicity (\cite{PT}). Therefore we have
 $mult(\ell'(\widehat{\gamma})) = k \ mult(\gamma)$.  Since the strict transforms of $\gamma$ and ${\gamma'}$ by $\rho_{\ell}$ and  $\rho_{\ell'}$  are curvettes of $C_{i}$ and $C'_{i}$ respectively and since these exceptional curves correspond to the same node $(i)$ of $T$,  then 
 $mult(\gamma) = mult(\gamma')$. We therefore obtain
 $mult(\ell'(\widehat{\gamma}))= k \ mult(\gamma') $. Since
 ${\gamma'}^*$ is a curvette of $C'_{i}$, then all the characteristic
 Puiseux exponents of $\gamma'$ are $\leq q_{C_{i}}$. Since the Puiseux
 expansions of $\gamma'$ and $\ell'(\widehat{\gamma})$ coincide up to
 exponent $q_{C_{i}}$, we then obtain that $\ell'(\widehat{\gamma})$
 admits a characteristic exponent $q > q_{C_{i}}$.  After change of
 coordinates if necessary, we can assume that $\ell = (x,y)$,
 $\ell' = (x,z)$ and that $\gamma$ is tangent to the $x$-axis. We 
  consider the real slices of the
 curves by intersecting them with $\{x=t \in \R^+\}$. Since $q$ is a characteristic exponent of
 $\ell'(\widehat{\gamma})$, there exists two real arcs
 $t \mapsto p_1(t) $ and $t \mapsto p_2(t) $ in the real slice of 
 $\ell'(\widehat{\gamma})$ such that
 $d_{o}(p_1(t),p_2(t)) = \Theta(t^{q})$. Lifting $p_1$ and $p_2$ by
 $\ell'$, we obtain two real arcs $t \mapsto \delta_1(t)$ and
 $t \mapsto \delta_2(t)$ inside the real slice of $\widehat{\gamma}$. Since
 $\ell'\mid_{\widehat{\gamma}} \colon \widehat{\gamma} \to
 \ell'(\widehat{\gamma})$ is a bilipschitz map for the outer metric,
 then $q_{o}(\delta_1,\delta_2) =q_{o}(p_1,p_2) = q$.
 
 We claim $\ell(\delta_1(t)) = \ell(\delta_2(t))$.  Indeed, assume the
 contrary and let $q'$ be the rational number defined by
 $d_o(\ell(\delta_1(t)) , \ell(\delta_2(t))) = \Theta(t^{q'})$.
 Since $\ell$ is the restriction of a linear projection, we
   have $q' \geq q$ and then $q' > q_{C_{\nu}}$. On the other hand,
 $\ell(\delta_1)$ and $ \ell(\delta_2)$ are distinct components of the
 real slice $\gamma \cap \{x=t\}$, therefore $q'$ is one of the
 characteristic exponents of $\gamma$. So we have $q' \leq q_{C_{i}}$. Contradiction.

 Since $\ell(\delta_1) = \ell(\delta_2)$, then $\delta_1$ and
 $\delta_2$ are tangent to the same semi-line and we get
 $q_{i}(\delta_1,\delta_2) = q_{C_{\nu}}$ by Lemma
 \ref{lem:inner X}.  Since $q>q_C$, then
 $q_{i}(\delta_1,\delta_2) < q_o (\delta_1,\delta_2)$ so the pair
 $(\delta_1,\delta_2) $ does not satisfy the arc criterion (Theorem
 \ref{thm:arc criterion}) and $(X,0)$ is not LNE. Contradiction.
\end{proof}

 The commutative diagram  $\cal E  \circ \widehat{\cal C}=\widehat{\cal E} \circ \cal C$ has the following remarkable property which will be used in
  the proof of the ``if'' direction of Theorem \ref{cor:complex characterization of normal
    embedding}
    
\begin{lemma}
 \label{rk:monodromy} Let $m$ be the multiplicity of $(X,0)$. There exists an integer $k>0$, such that 
 \begin{enumerate}
\item  ${\cal E}\colon \cal F \to T$ is the quotient map of a  cyclic action of $\Z/k\Z$ on the graph $  {\cal F}$;
\item   $\widehat{\cal E}\colon \widehat{\cal F} \to \cal G$ is the quotient map of a  cyclic action of $\Z/km\Z$ on the graph $\widehat {\cal F}$;
\item the graph-map $\widehat{\cal C} \colon  \widehat {\cal F}  \to  {\cal F} $ is equivariant for these actions. 
\end{enumerate}
\end{lemma}

 \begin{proof} Set $\ell=(x,y)$ where $h=x$ is a generic linear form on $(X,0)$. By the method described in \cite[Section 1]{DBM}, one can construct, using the resolution $\rho_{\ell}$, a quasi-periodic representative  $\phi \colon {F}_t \to {F}_t $  of the monodromy of the Milnor fibration of $h$ which has the following property:  for each vertex $(v)$ of $T$, $\phi$ exchanges cyclically the connected components of
  $B_v \cap {F}_t$, for each edge $(v)-(v')$ of $T$,   $\phi$ exchanges cyclically the annular connected components of
  $A_{v,v'} \cap {F}_t$, and there is an integer  $k >0$ such that $\phi^k$ is the identity on $B_v \cap {F}_t$ and a Dehn twist on each annulus component of $A_{v,v'} \cap {F}_t$. This action induces a  cyclic action of $\Z/k\Z$ on the graph $  {\cal F} $  whose quotient map  is ${\cal E}\colon \cal F \to T$. 
  
 This method  to construct $\phi$ is presented in \cite{DBM} in the case of the Milnor
  fibration of a holomorphic function germ $f \colon (\C^2,0) \to (\C,0)$, but the same method applies in 
  the more general setting of a holomorphic function
  $ (X,0) \to (\C,0)$ where $(X,0)$ is a normal surface
  germ. Applying this method to  the holomorphic function $h \circ \ell \colon (X,0) \to (\C)$ and to the resolution $\pi_{\ell} \colon  X'_{\ell} \to X$ (where $h=x$), we construct a  quasi-periodic  representative  $\phi \colon \widehat{F}_t \to \widehat{F}_t$ with order $km$ of the monodromy of  $h \circ \ell $ which induces a  cyclic action of $\Z/km\Z$ on the graph $  \widehat{\cal F} $  whose quotient map  is $\widehat{\cal E}\colon \widehat{\cal F} \to \cal G$. 
  
By construction,  the graph-map $\widehat{\cal C} \colon  \widehat {\cal F}  \to  {\cal F} $ is equivariant for these actions. 
 \end{proof}

 \begin{proof}[of the ``if'' direction of Theorem 
  \ref{cor:complex characterization of normal embedding}]
  Assume that Conditions ($1^*$) and ($2^*$) are satisfied (we will prove at the end that Condition  ($2^*$) is not needed for test curves at $\Delta$-nodes).  First, notice that  Condition  ($1^*$) implies that for all nodal test curve $\gamma$ and all principal component $\widehat{\gamma}$ of $\ell^{-1}(\gamma)$, the restriction $\ell \mid_{\widehat{\gamma}} \colon \widehat{\gamma} \to \gamma$ has degree $1$, i.e., it is bijective. Let  $\delta$  be a nodal test arc and let  $\delta_1$ and $\delta_2$ be  two distinct  principal components of $\ell^{-1}(\delta)$. Let $\gamma$ be the nodal test curve such that $\delta$ is a real slice of $\gamma$ and let $\gamma_1$ and $\gamma_2$ be the two components of $\ell^{-1}(\gamma)$ containing respectively $\delta_1$ and $\delta_2$.   Let $(\nu_1)$ and $(\nu_2)$ be the two vertices of $\widehat{\cal F} $ such that $\delta_i(t) \in \widehat{F}_{t,\nu_i}$.

\vskip0,2cm\noindent   
{\bf Case 1.} Assume  that $q_{\nu_1, \nu_2} = q_j$. This implies  $q_i(\delta_1, \delta_2) = q_j$ by Lemma  \ref{lem:inner X}. Let $\delta'_1$ (resp.\ $\delta'_2$) be a component of the real slice $\gamma_1 \cap \{x=t\}$ (resp.\ $\gamma_2 \cap \{x=t\}$).  

Assume  that  $\ell(\delta'_1) = \ell(\delta'_2)$ and let $(\nu'_1)$ and $(\nu'_2)$ be the two vertices of $\widehat{\cal F} $ such that $\delta'_i(t) \in \widehat{F}_{t,\nu'_i}$.  By condition ($1^*$), if $\delta'_1$ is a real slice  $\gamma_1\cap \{x=t\}$, there is a unique component $\delta'_2 $ of the real slice  $\gamma_2 \cap \{x=t\}$ such that $\ell(\delta'_1) = \ell(\delta'_2)$, and by Lemma \ref{rk:monodromy}, we have $q_{\nu_1, \nu_2} =  q_{\nu'_1, \nu'_2}$. Therefore,    $q_i(\delta'_1, \delta'_2) =  q_i(\delta_1, \delta_2) =q_j$ by Lemma \ref{lem:inner X}.

 Assume now that $\ell(\delta'_1) \neq \ell(\delta'_2)$.  Let $S_t$ be
the segment in $\C^2$ joining $\ell(\delta'_1)$ to $\ell(\delta'_2)$
and let $\delta''_2(t)$ be the extremity of the lifting of $S_t$ with
origin $\delta'_1(t)$. Then $\delta''_2$ is a component of the real
slice $\gamma_2 \cap \{x=t\}$ such that
$\ell(\delta''_2) = \ell(\delta'_2)$, and using the arguments of the
proof of Lemma \ref{liftingoftestarcs}, we have
$q_o(\delta'_1, \delta'_2) = \min(q_i(\delta'_1, \delta''_2),
q_o(\delta''_2, \delta'_2))$.
Then a fortiori
$q_i(\delta'_1, \delta'_2) \leq q_o(\delta'_1,\delta'_2)\leq  q_i(\delta'_1, \delta''_2)$. Since
$\ell(\delta'_1)$ and $\ell(\delta''_2) = \ell(\delta'_2)$ are
distinct component of a real slice of $\gamma$, then
$q_i(\ell(\delta'_1), \ell(\delta'_2))$ equals one of the essential
Puiseux exponents of $\gamma$, which are all $\leq q_j$. Since 
$q_i(\delta'_1, \delta''_2) \le q_i(\ell(\delta'_1),
\ell(\delta''_2))$, then 
$q_i(\delta'_1, \delta'_2) \leq q_i(\ell(\delta'_1), \ell(\delta'_2))$,
so we then obtain $q_i(\delta'_1, \delta'_2) \leq q_j$.  Therefore,
$q_{inn}(\gamma_1, \gamma_2) = q_i(\delta_1, \delta_2)$ (cf.\ Proposition \ref{lem:complex real rates}). Using Condition $(2*)$, we then have
$q_o(\delta_1, \delta_2) \leq q_{out}(\gamma_1, \gamma_2) = q_{inn}(\gamma_1,\gamma_2) = q_i(\delta_1, \delta_2)$.
Therefore $q_o(\delta_1, \delta_2) \leq q_i(\delta_1, \delta_2)$, and
then $q_o(\delta_1, \delta_2)= q_i(\delta_1, \delta_2)$ since the converse inequality is always true.

\vskip0,2cm\noindent
{\bf Case 2.} Assume now that $q_{\nu_1\nu_2}<q_j$ and consider a simple path in $\widehat {\cal F}$ from $(\nu_1)$ to $(\nu_2)$ such that $q_{\nu_1, \nu_2}$  is the minimal inner rate  along it and let  $(\nu)$ be the vertex on it such that $q_{\nu}=q_{\nu_1, \nu_2}$. Set $(j') = (\cal E \circ \widehat{\cal C})(\nu)$. For any test curve $\gamma'$ at $(j')$ and any test arc $\delta' \subset \gamma'$, we can apply Case 1 to any pair of $\delta'_1, \delta'_2$  over $\delta'$ such that $\delta'_1(t)$ and  $\delta'_2(t)$ are in $\widehat{\cal F}_{t,\nu}$ and we get $q_i(\delta'_1, \delta'_2) = q_o(\delta'_1, \delta'_2)$. Then using Lemma \ref{lem:partner3}, we conclude  $q_i(\delta_1, \delta_2) = q_o(\delta_1, \delta_2)$.
\end{proof}

 \section{Inner contacts on a normal surface and resolution}
 
 In this last section, we state and prove an analog of  Lemma \ref{lem:inner X} in terms of complex curves and resolution which is very useful to compute inner rates between complex curves  on a normal surface germ using resolution and then to  check  condition ($2^*$) of Theorem  \ref{cor:complex characterization of normal embedding}.  First, let us introduce a generalization of the inner rates of Definition \ref{def:inner rate}. 

\begin{lemma} \label{rk:inner rate} Let $\pi \colon X' \to X$ be a resolution of $X$ and let  $E$ be an irreducible component of the exceptional divisor $\pi^{-1}(0)$. Let $\gamma$ and $\gamma'$ be two complex curve germs in $(X,0)$ whose strict transforms by $\pi$ are curvettes of $E$ meeting $E$ at two distinct points.  Then  $q_{inn}(\gamma, \gamma')$  is independent of the choice of $\gamma$ and $\gamma'$.
  \end{lemma}
  
  \begin{definition}  \label{def:inner rate2} We set $q_E = q_{inn}(\gamma, \gamma')$ and we call $q_E$ the inner rate of $E$. 
 \end{definition}
 
 \begin{proof} Let $\ell \colon (X,0) \to (\C^2,0)$ be a generic projection for $(X,0)$ which is also generic for the curve $\gamma \cup \gamma'$  and let $\rho \colon Y \to \C^2$ be the minimal sequence of blow-ups defined as in the proof of Proposition \ref{lem:complex real rates}.  There is a component  $C$ of $\rho^{-1}(0)$ which intersect both  $\ell(\gamma)^*$ and $\ell(\gamma')^*$.  Let $q_C$ be its inner rate. Then the piece $\rho(\cal N(C))$ of the associated geometric decomposition of $\C^2$ lifts by $\ell$ to a union of $B(q_C)$-pieces, one of them being  of the form $\pi(\cal N (E))$. Therefore, $q_{inn}(\gamma, \gamma')=q_C$.   \end{proof}

We then have the following consequence of Proposition \ref{lem:complex real rates} and  Lemma \ref{lem:inner X}. 

 \begin{proposition} \label{lem:inner contact computation} Let $\gamma$ and $\gamma'$ be two complex curves on $(X,0)$.  Consider a  resolution  $\pi \colon X' \to X$    which factors through the Nash modification and through the blow-up of the maximal ideal  and which is a resolution of the complex curve $\gamma \cup \gamma'$ and set $\pi^{-1}(0) = \bigcup_{v}E_v$.  Let $\widetilde{G}$ be the resolution graph of $\pi$ whose vertices $(v)$ are weighted by the inner rates introduced in Definition \ref{def:inner rate2}. Let  $(v)$ and $(v')$ be the vertices of $G$ such that $\gamma^* \cap E_{v} \neq \emptyset$ and    ${\gamma'}^* \cap E_{v'} \neq \emptyset$.  Then $q_{inn}(\gamma, \gamma')= q_{v,v'}$ where   $q_{v,v'}$ is the   maximum among minimum of inner rates  along paths from $(v)$ to $(v')$ in the graph $\widetilde{G}$. 
 \end{proposition}
 
  \begin{proof} By Lemma \ref{lem:complex real rates}, we have $q_{inn}(\gamma, \gamma') = \max_{k,l}  q_i(\delta_k,  {\delta'_l})$ where the $\delta_k$ (resp.\ $\delta'_l$) are the components of a real slice of $\gamma$ (resp.\ $\gamma'$). 
  
First, notice that the statements and proofs of Lemmas \ref{lem:inner X} and  \ref{rk:monodromy} stay the same if one replaces the resolution $\rho_{\ell}$  and the associated geometric decompositions  of $(\C^2,0)$, $(X,0)$, $\cal F$ and $\widehat{\cal F}$ by any  sequence of blow-ups of points $\rho \colon Y \to \C^2$  which resolves the basepoints of the family of projected polar curves $(\ell(\Pi_{\cal D}))_{\cal D \in \Omega}$.  As in the proof of Proposition \ref{lem:complex real rates}, consider  a generic  projection $\ell \colon (X,0) \to (\C^2)$ which is also generic for the curve $\gamma \cup \gamma'$ and let $\rho \colon Y \to \C^2$ be the minimal sequence of blow-ups of points which resolves the basepoints of the family of projected polar curves $(\ell(\Pi_{\cal D}))_{\cal D \in \Omega}$ and which resolves the curve $\ell(\gamma) \cup \ell(\gamma')$. We will use  the same  notations as before for the graphs  of the associated geometric decompositions, in particular the graph-map $\widehat{\cal E} \colon \widehat{\cal F} \to  \cal G$.    If $(v)$ and $(v')$ are two vertices of $\cal G$, we will denote by $q_{v,v'}$ the maximal among minimal rates along simple paths in $\cal G$ between the vertices $(v)$ and $(v')$. 
Since $\rho$ resolves the curve $\ell(\gamma) \cup \ell(\gamma')$, then there are two vertices $(v)$ and $(v')$ in $\cal G$ such that $\gamma$ is contained in the $B$-piece $B_v$ and $\gamma'$ is contained in the $B$-piece $B_{v'}$.
  As a consequence of Lemma \ref{rk:monodromy}, the map $\widehat{\cal E} \colon \widehat{\cal F} \to \cal  G$ satisfies the following property: if $p$ is a simple path in $\cal G$ from the vertices $(v)$ to $(v')$ and if $(\nu)$ and $(\nu')$ are two vertices in $\widehat{\cal F}$ such that $\widehat{\cal E} (\nu) = (v)$ and $\widehat{\cal E} (\nu') = (v')$, then there exists $m \in \N^*$ such that the lifting by $\widehat{\cal E}$ of $mp$ with origin $(\nu)$ is a simple path $\widehat{p}$  from $(\nu)$ to $(\nu')$ in  $\widehat{\cal F}$.
 
  Let  $p$ be a path in $ \cal G$ between the vertices $(v)$ and $(v')$ such that the minimal inner rate along $p$ equals $q_{v,v'}$. Consider two  components $\delta \subset \gamma$ and $\delta' \subset \gamma'$ of the real slices. Let $(\nu)$ and $(\nu')$ be the two vertices of $\widehat{\cal F}$ such that $\delta(t) \in \widehat{F}_{t,\nu}$ and $\delta'(t) \in \widehat{F}_{t,\nu'}$.  Lemma \ref{lem:inner X} extends with the same statement and the same proof to a computation of inner contacts using the new graph $\widetilde{\cal F}$ instead of $\cal F$: $q_i(\delta, \delta')$ equals $q_{\nu, \nu'}$ where $q_{\nu, \nu'}$ equals the maximum of minimum inner rates among all simple paths in $\widetilde{\cal F}$ from $(\nu)$ to $(\nu')$. Using the remark above, we obtain $q_{v,v'} \leq q_{\nu, \nu'}$. 
 
 Now,  take a pair  $\delta, \delta'$ such that $q_i(\delta, \delta') = q_{inn}(\gamma, \gamma')$  (Lemma \ref{lem:complex real rates}) and take a path $\widehat{p}$ in  $\widehat{\cal F}$ which realizes $q_{\nu, \nu'}$. Then $\widehat{p}$ projects by $\widehat{\cal E} $ to a path $p$ whose support is a simple path from $(v)$ to $(v')$. Therefore $q_{\nu, \nu'} \leq q_{v,v'}$, and the previous inequality gives then $q_{\nu, \nu'} = q_{v,v'}$.  Then applying Lemma \ref{lem:complex real rates}, we obtain $q_{inn}(\gamma, \gamma') = q_{v,v'}$, where this number is computed in the graph $\cal G$. 
 
As in Section \ref{subsec:cal G-resolution}, there is again a natural  injection $\cal I \colon V(\cal G) \to V(\widetilde{G})$. We denote again by $(v)$ and $(v')$ the images $\cal I (v)$ and $\cal I (v')$.   By construction, the  sequence of inner rates along a string in $\widetilde{G}$ between two consecutive vertices $\cal I(v)$ and $\cal I(v')$ is strictly monotone. Therefore,  $q_{v,v'}$  computed in the graph $\cal G$ equals the  number $q_{v,v'}$ computed in the graph $\widetilde G$. This proves the proposition.  \end{proof}

\affiliationone{Walter D Neumann\\
Department of Mathematics\\
Barnard College, Columbia University\\
2990 Broadway MC4429\\
New York, NY  10027\\ USA \email{neumann@math.columbia.edu}
} 
\affiliationtwo{Helge M\o ller Pedersen\\
Departamento de Matem\'atica\\ Universidade Federal do Cear\'a\\
Campus do Pici, Bloco 914 \\ CEP 60455-760 \\ Fortaleza, CE \\Brazil \email{helge@mat.ufc.br}}
\affiliationthree{Anne Pichon\\
Aix Marseille Universit\'e, CNRS\\ 
Centrale Marseille, I2M, UMR 7373\\ 
13453 Marseille\\ France \email{anne.pichon@univ-amu.fr}}

\end{document}